\documentclass[12pt]{amsart}
\usepackage{amsmath,amssymb,fullpage,amsrefs}
\usepackage[hidelinks, colorlinks=true, linkcolor=blue, citecolor=magenta]{hyperref}
\usepackage{verbatim}
\usepackage{listings}
\newcommand{\sA}{\mathcal{A}}
\newcommand{\sP}{\mathcal{P}}
\newcommand{\sQ}{\mathcal{Q}}
\newcommand{\sW}{\mathcal{W}}

\newcommand{\CC}{\mathbb{C}}
\newcommand{\FF}{\mathbb{F}}
\newcommand{\PP}{\mathbb{P}}
\newcommand{\QQ}{\mathbb{Q}}
\newcommand{\RR}{\mathbb{R}}
\newcommand{\ZZ}{\mathbb{Z}}
\newcommand{\Qbar}{\overline{\QQ}}
\newcommand{\ty}{\tilde{y}}
\newcommand{\sep}{\mathrm{sep}}

\newcommand{\sJ}{\mathcal{J}}
\newcommand{\sK}{\mathcal{K}}
\DeclareMathOperator{\He}{He}
\DeclareMathOperator{\Gal}{Gal}
\DeclareMathOperator{\Aut}{Aut}
\DeclareMathOperator{\Pic}{Pic}
\DeclareMathOperator{\Div}{Div}
\DeclareMathOperator{\EC}{EC}
\DeclareMathOperator{\Res}{Res}
\DeclareMathOperator{\PSp}{PSp}
\DeclareMathOperator{\PGL}{PGL}
\DeclareMathOperator{\Sp}{Sp}
\DeclareMathOperator{\rk}{rk}

\newtheorem{theorem}{Theorem}[section]
\newtheorem{prop}[theorem]{Proposition}
\newtheorem{lemma}[theorem]{Lemma}
\newtheorem{cor}[theorem]{Corollary}
\theoremstyle{definition}
\newtheorem{dfn}[theorem]{Definition}
\newtheorem{remark}[theorem]{Remark}
\newtheorem{question}[theorem]{Question}

\begin{document}
\title{Arithmetic aspects of the Burkhardt quartic threefold}
\date{April 11, 2018}
\author{Nils Bruin}
\address{Department of Mathematics, Simon Fraser University, Burnaby, BC, V5A 1S6, Canada}
\email{nbruin@sfu.ca}
\thanks{This work is partially funded by NSERC}
\author{Brett Nasserden}
\address{Department of Pure Mathematics, University of Waterloo, Waterloo, ON, N2L 3G1, Canada}
\email{bnasserd@uwaterloo.ca}
%14H10   	Families, moduli (algebraic)
%14K10   	Algebraic moduli, classification
%11G10   	Abelian varieties of dimension $> 1$ 
%11G18   	Arithmetic aspects of modular and Shimura varieties

\subjclass[2010]{11G10, 14K10, 11G18, 14H10}
\keywords{Genus 2 curves, Abelian Varieties, Moduli spaces, Level structure}

\begin{abstract}
We show that the Burkhardt quartic threefold is rational over any field of characteristic distinct from 3. We compute its zeta function over finite fields. We realize one of its moduli interpretations explicitly by determining a model for the universal genus $2$ curve over it, as a double cover of the projective line. We show that the $j$-planes in the Burkhardt quartic mark the order $3$ subgroups on the Abelian varieties it parametrizes, and that the Hesse pencil on a $j$-plane gives rise to the universal curve as a discriminant of a cubic genus one cover.
\end{abstract}
\maketitle
\section{Introduction and results}
We consider the \emph{Burkhardt quartic threefold} in $\PP^4$, defined by the equation
\[B\colon f(y_0,\ldots,y_4):=y_0(y_0^3+y_1^3+y_2^3+y_3^3+y_4^3)+3y_1y_2y_3y_4=0.\]
This threefold has been studied extensively over $\CC$ and can be characterized in many ways.
\begin{enumerate}
\item It has a linear action of the finite simple group $\PSp_4(\FF_3)$. In fact it is defined by the unique quartic invariant of this linear representation.
\item It has $45$ nodal singularities, which is the maximum for a quartic threefold \cite{Varchenko83}. Furthermore, up to projective equivalence it is the \emph{only} one \cite{JongShepVen90}.
\item It has various interpretations in terms of moduli spaces \cites{Elkies99, vGeemen92, Hunt96}. Most important for us is that it is birational to the moduli space $\sA_{2,3}$ of Abelian surfaces equipped with a full level-$3$ structure (see Definition~\ref{D:level3structure})\footnote{More precisely, the normalization of the projective dual is the Satake compactification of $\sA_{2,3}$, see \cite{FreitagSalvati04}.}. This in fact holds over $\ZZ[1/3,\zeta_3]$, see \cites{Hunt96,HuntWeintraub94,vdGeer87}.
\end{enumerate}
The geometry of the Burkhardt quartic gives rise to various intricate combinatorial configurations that have been been extensively studied \cite{Baker46} (see \cite{Hunt96} for a modern account and \cite{RenSamSturmfels14} for a modern, tropical description).

For arithmetic applications one also needs to consider $B$ over base fields that are not algebraically closed. For instance, Todd \cite{Todd1936} proved that $B$ is rational over $\CC$ and in 1942 Baker \cite{Baker46}*{\S6} exhibited an explicit parametrization defined over $\QQ(\zeta_3)$, but Baker's parametrization does not naturally descend to $\QQ$. We show that, with some different choices, it does. This also provides us with an easy way to determine the zeta function of $B$ over any finite field of characteristic different from $3$, generalizing results in  \cite{HoffmanWeintraub01} for fields of cardinality $q\equiv1\pmod{3}$.

Other questions arise from the modular interpretation of $B$. An open part of $\sA_{2,3}$, isomorphic to an open part of $B$, corresponds to Jacobians of genus $2$ curves, so one expects there to be a universal genus $2$ curve $C_\alpha$, defined over that open part such that its Jacobian $\sJ_\alpha$ realizes the moduli interpretation. Such a curve should admit a model as a double cover of $\PP^1$, ramified at $6$ points. The geometric moduli for this are given in \cite{Hunt96}, but the field of moduli of a genus $2$ curve famously doesn't need to agree with its field of definition. In this case, the fact that $\sA_{2,3}$ is a fine moduli space guarantees the obstruction is trivial, but explicitly showing this requires work.

Naturally, the moduli interpretation also implies that $C_\alpha$ should come equipped with divisor classes of order $3$, marking the level structure on its Jacobian. We explicitly determine how these arise from the geometry of $B$.

We also show how $C_\alpha$ can be obtained from the degree $6$ branch locus of certain cubic genus $1$ covers of $\PP^1$ that can be directly constructed from a point $\alpha\in\PP^1$ using the geometry of $B$.

Section~\ref{S:results} below states the results while Sections~\ref{S:param}-\ref{S:moduli} provide the proofs. Appendix~\ref{Appendix} contains most relevant formulae in a computer-readable form. These formulae are also available in electronic form from \cite{BruinNasserden17e}.

\section{Statement of Results}
\label{S:results}
In Section~\ref{S:param} we adapt Baker's parametrization to descend to $\QQ$. We obtain the following.

\begin{theorem}\label{T:ratpar}
Let $k$ be a field of characteristic not equal to $3$. Then $B$ is birational to $\PP^3$ over $k$ by the map
$\psi\colon B\to\PP^3$;
$(y_0:y_1:y_2:y_3:y_4)\mapsto (t_0:t_1:t_2:t_3)$, where
\[
\begin{aligned}
t_0&= y_0(y_0^2-y_0y_1+y_1^2),\\
t_1&= y_0(y_1y_2-y_0y_3-y_0y_4),\\
t_2&= y_0(y_0y_2 - y_1y_2 + y_1y_3 + y_1y_4),\\
t_3&= y_0y_1y_2 - y_0y_1y_3 + y_1^2y_3 - y_0^2y_4.
\end{aligned}
\]
\end{theorem}

In Section~\ref{S:zeta} 
we use this map and the parametrization inverse to it for computing the zeta function of $B$ over arbitrary finite fields $\FF_q$ of characteristic not $3$. It may be interesting to compare with the approach in \cite{HoffmanWeintraub01} where a fibration of $B$ is used to compute the zeta function for $q\equiv 1 \pmod{3}$.

\begin{theorem}\label{T:zetaB}
Let $q$ be a prime power not divisible by $3$, and let $\epsilon=\left(\frac{q}{3}\right)$. Then
\[Z(B/\FF_q,T)=\dfrac{(1-qT)^{15}(1-\epsilon qT)^{14}}{(1-T)(1-q^2T)^{10}(1-\epsilon q^2T)^6(1-q^3T)}.\]
\end{theorem}
\begin{cor}\label{C:zetaBdesing}
With the notation above, let $\tilde{B}$ be the desingularization of $B$ obtained by blowing up the singularities on $B$. Then
\[Z(\tilde{B}/\FF_q,T)=\dfrac{1}{(1-T)(1-qT)^{36}(1-\epsilon qT)^{25}(1-\epsilon q^2T)^{25}(1-q^2T)^{36}(1-q^3T)}
\]
\end{cor}
It is known that the complement of the \emph{Hessian} $\He(B)$ on $B$ is isomorphic to the part of $\sA_{2,3}$ that parametrizes Jacobians of genus $2$ curves. By computing the zeta function of $B\cap\He(B)$ as well, we find the following.

\begin{cor}\label{C:pointcount}
\[\#(B\setminus \He(B))(\FF_q)=\begin{cases}
(q-4)(q-7)(q-13)&\text{if }q\equiv 1\pmod{3}\\
(q-2)(q^2-2q-1)&\text{if }q\equiv 2\pmod{3}
\end{cases}\]
\end{cor}

We see there are no genus $2$ curves that have Jacobians with full $3$-torsion over $\FF_4,\FF_7,\FF_{13}$, and that therefore any genus $2$ curve $C$ over $\QQ$ that has $J_C[3]\simeq (\ZZ/3)^2\times\mu_3^2$  must have bad reduction at $2,7,13$. We make no claim about the reduction of their Jacobians. Note that there are genus $2$ curves over $\FF_7,\FF_{13}$ for which the divisor class groups of degree $0$ have cardinality $81$.

The fact that $\sA_{2,3}$ is a fine moduli space also implies there exists a \emph{universal genus 2 curve} $C_\alpha$ defined over $B\setminus \He(B)$ such that its Jacobian $\sJ_\alpha$ has a full level-$3$ structure on it. A level $3$-structure for us is an isomorphism $(\ZZ/3)^2\times\mu_3^2\to\sJ_\alpha[3]$ as group schemes equipped with alternating pairing.
Hunt \cite{Hunt96} describes the data defining such a curve geometrically in the form of a plane conic with $6$ marked points on it, but that does not immediately lead to a model defined over the base field (see for instance \cite{Mestre91}).

He also provides a model for the variety representing $\Pic^1(C_\alpha)$ in $\PP^8$. This gives a certificate that the universal curve can indeed be defined over the base field, but extracting a model as a double cover of $\PP^1$ is not entirely straightforward. In Section~\ref{S:explicit_curve} we do this using the classical theory of Weddle and Kummer surfaces and find the following model.

\begin{prop}\label{P:explicit_curve}
Let $\alpha=(1:\alpha_1:\cdots:\alpha_4)\in B\setminus\He(B)\setminus\{x_4=0\}$. Then $\sJ_\alpha$ arises as the Jacobian of the hyperelliptic curve
\[y^2+G_3y=\lambda_3H_3^3\]
where
\[
\begin{split}
H_3&=\alpha_2 x^2 - \alpha_3 xz - \alpha_1 \alpha_4 z^2,\\
G_3&=(\alpha_1^3 \alpha_4^3 + 3 \alpha_1 \alpha_2 \alpha_3 \alpha_4^4 + 2 \alpha_2^3 \alpha_4^3 + \alpha_2^3 + \alpha_3^3 \alpha_4^3) x^3\\
   &\quad+3 \alpha_2(\alpha_4^3+1) (\alpha_1^2 \alpha_4^2-\alpha_2 \alpha_3) x^2z -3 \alpha_3 (\alpha_4^2+1) (\alpha_1^2 \alpha_4^2-\alpha_2 \alpha_3) xz^2\\
   &\quad+ (-2 \alpha_1^3 \alpha_4^6 - \alpha_1^3 \alpha_4^3 + 3 \alpha_1 \alpha_2 \alpha_3 \alpha_4^4 + \alpha_2^3 \alpha_4^3 - \alpha_3^3)z^3,\\
\lambda_3&=\alpha_4^3 (\alpha_4^3+1) (\alpha_1 \alpha_4-\alpha_2-\alpha_3)(\alpha_1^2 \alpha_4^2 + \alpha_1 \alpha_2 \alpha_4 + \alpha_1 \alpha_3 \alpha_4 + \alpha_2^2 - \alpha_2 \alpha_3 + \alpha_3^2).
\end{split}
\]
\end{prop}

While the theory of Weddle and Kummer surfaces requires the base field to be not of characteristic $2$, we can extend our model to be over $\ZZ[1/3]$ and check it has good reduction at $2$ as well, and argue by specialization.

In Section~\ref{S:marking} we consider how to explicitly mark the level-$3$ structure on $C_\alpha$. For this we use the Kummer surface $\sK_\alpha=\sJ_\alpha/\langle-1\rangle$, which has a natural model in $\PP^3$, as well as its projective dual $\sK_\alpha^*$, which is isomorphic to $\sK_\alpha$ over an algebraically closed base field, but not in general. Following a classical construction (see for instance \cite{Coble1917}*{p.~360}), Hunt describes how $\sK_\alpha^*$ can be obtained as the image under a projection $\pi_\alpha\colon \PP^4\to \PP^3$ of the enveloping cone of the cubic polar of $B$ at $\alpha$ (see for instance \cite{Dolgachev12}*{\S1.1} for definitions of these).

The classical combinatorics of $B$ shows that $B\cap\He(B)$ consists of $40$ planes, each containing $9$ of the singularities of $B$. These planes are classically referred to as $j$-planes. Furthermore, there are $40$ hyperplanes that intersect $B$ in the union of $4$ $j$-planes, called Steiner primes. Conversely, every $j$-plane lies in $4$ Steiner primes.

\begin{prop}\label{P:Jplane_3torsion}
Let $\alpha\in B\setminus \He(B)$, let $\pi_\alpha\colon \PP^4\to \PP^3$ be the projection from $\alpha$, and let $\sK_\alpha^*\subset\PP^3$ be the dual Kummer surface obtained by projecting the enveloping cone of the cubic polar of $B$ at $\alpha$.
\begin{enumerate}
\item If $J$ is a $j$-plane, then $\pi_\alpha(J)$ is tangent to $\sK_\alpha^*$, and hence a point on $\sK_\alpha$.
\item The point on $\sK_\alpha$ determined by $J$ lifts to $3$-torsion points on $\sJ_\alpha$.
\item Two $3$-torsion points on $\sJ_\alpha$ pair trivially under the Weil pairing if and only if they are coming from $j$-planes that lie in a common Steiner prime.
\item Hence, Steiner primes correspond to the maximal isotropic subgroups of $\sJ_\alpha[3]$.
\end{enumerate}
\end{prop}

We use that non-principal degree $0$ divisor classes on genus $2$ curves can be represented uniquely by $[D-\kappa]$, where $D$ is an effective divisor of degree $2$ and $\kappa$ is a canonical divisor. The following geometric description of the relation between points on the Kummer surface and the divisor $D$ corresponding to it turned out useful, and we were unable to find it elsewhere in the literature.

\begin{prop}\label{P:kummer_find_intersection}
Let $C$ be a curve of genus $2$ over a field of characteristic different from $2$ and let $[D-\kappa]\in\Pic^0(C)$ be a divisor class represented by the effective divisor $D\in\Div^2(C)$. Let $T_D$ be the tangent plane to the dual Kummer surface $\sK_C^*$ and let $L$ be the conic cut out on $\sK_C^*$ by the distinguished trope on $\sK_C^*$ corresponding to the image of the identity element of $\sK_C$. Then $L\simeq\PP^1$, and the hyperelliptic cover $C\to\PP^1$ is naturally realized as $x\colon C\to L$, with the $6$ ramification points being the $6$ nodes of $\sK_C^*$ that $L$ passes through.
We have
\[x_*(D)=T_D\cdot L.\]
\end{prop}

With this result it is straightforward, given a $j$-plane, to get a representing divisor and check that $3(D-\kappa)$ is a principal divisor. If $x_*(D)$ is defined by a quadratic equation $H(x)=0$ on $\PP^1$, then a certificate of this principality is given by the existence of $\lambda$ and a cubic $G(x)$ such that $y^2+G(x)y=\lambda H(x)^3$ is a model of the curve. If $2$ is invertible, then this is equivalent to a model of the form $y^2=G(x)^2+4\lambda H(x)^3$.
For the $j$-plane $J_1\colon y_0=y_1=0$ we write $y^2+G_1(x)y=\lambda_1H_1(x)^3$ for the model thus obtained.

For the order $3$ subgroups of the form $\mu_3$ (which has a different Galois structure than $\ZZ/3$ if the base field does not contain the cube roots of unity) we find a twisted model $y^2+G(x)y+G(x)^2=-3\lambda H(x)^3$ (which is isomorphic to $-3y^2=G(x)^2+4\lambda H(x)^3$ if $2$ is invertible). 

Baker and Hunt also remark that the cubic polar $P^{(1)}_\alpha(B)$ of $B$ at $\alpha$ describes a Hesse pencil on each $j$-plane. This associates a cubic curve $E_{J,\alpha}$ to $\alpha$. In fact, these curves arise as subcovers of the unramified Abelian cubic cover of $C_\alpha$ determined by the order $3$ subgroup marked by $J$. Conversely, it means we can recover $C_\alpha$ (up to quadratic twist) from the discriminant of a cubic genus $1$ cover of $\PP^1$. In particular, we show the following in Section~\ref{S:disc}; see Remark~\ref{R:cubic_cover_coordinate_free} for a coordinate-free description.

\begin{prop}\label{P:disc}
Let $\alpha=(\alpha_0:\cdots:\alpha_4)$ be a point on the Burkhardt quartic $B$. Then the intersection of the cubic polar $P_\alpha^{(1)}$ with the $j$-plane $J\colon y_0=y_1=0$ yields the plane cubic
\[E_{J,\alpha}\colon \alpha_0(y_2^3+y_3^3+y_4^3)+3\alpha_1y_2y_3y_4=0.\]
The cover $E_{J,\alpha}\to\PP^1$ obtained from projecting from $(y_2:y_3:y_4)=(\alpha_2:\alpha_3:\alpha_4)$ is equivalent to
\[w^3+3\lambda_1H_1(x,z)w+\lambda_1G_1(x,z)=0.\]
The curve $C_\alpha$ (up to quadratic twist) arises as the discriminant of $E_{J,\alpha}\to\PP^1$, and the fiber product $C_\alpha\times_{\PP^1} E_{J,\alpha}$ is the unramified cover of $C_\alpha$ that capitalizes the order $3$ subgroup of $\sJ_\alpha$ determined by $J$.
\end{prop}

The description of $C_\alpha$ as arising from a discriminant immediately exhibits it as a cover of $\PP^1$, avoiding the parametrization constructed in Section~\ref{S:explicit_curve}. However, to ensure that the curve matches up with the moduli interpretation we do require at least some information from Proposition~\ref{P:explicit_curve}.

Our final observations are on another classical model for $B$, obtained by setting $\sigma_1=\sigma_4=0$, where $\sigma_i$ is the $i$-th elementary symmetric polynomial in $6$ variables. This gives a more symmetric quartic model $B'\subset \PP^4$, embedded in $\PP^5$. As is easily checked, $B$ and $B'$ are isomorphic over fields containing the cube roots of unity. In other cases, however, $B'$ is a nontrivial twist of $B$. For instance, for $B'$ all the $j$-planes come in conjugate pairs.

It raises the question what level-3 structure is parametrized by $B'$. Given a point $\alpha$ on $B'$ we can obtain $6$ points on a conic in exactly the same way as for $B$. However, we find that if we take $\alpha\in B'(\RR)$ then the conic has no real points. Thus the moduli space parametrizes Kummer surfaces with level-$3$ structure that are \emph{not} a quotient of an Abelian variety defined over $\RR$.

We also note that the parametrization idea of Baker cannot be adapted to $B'$ over $\QQ$, so as far as we know it is still unknown if $B'$ is rational over $\QQ$. Indeed, the birational parametrization of $B$ does not arise from a construction that is particularly compatible with the modular interpretation of $B$ (see Remark~\ref{R:parmsym}). There are many twists of $B$, corresponding to the various full level-$3$ structures that can arise on Abelian surfaces. 

\begin{question} Which twists of $B$ are rational over $\QQ$?
\end{question}

\section{Some basic properties of the Burkhardt quartic}

The action of $\PSp_4(\FF_3)$ on $B$ is given by the right action on the row vector $(y_0,\ldots,y_4)$ by the matrices
\[
-\begin{pmatrix}
	1 & 0 & 0 & 0 & 0 \\
	0 & 1 & 0 & 0 & 0 \\
	0 & 0 & 0 & 0 & 1 \\
	0 & 0 & 0 & 1 & 0 \\
	0 & 0 & 1 & 0 & 0
\end{pmatrix},
\frac{1}{3}\begin{pmatrix}
	1 & 2 & 2 & 2 & 2 \\
	1 & -1 & -1 & 2 & -1 \\
	1 & -1 & -1 & -1 & 2 \\
	1 & -1 & 2 & -1 & -1 \\
	1 & 2 & -1 & -1 & -1
\end{pmatrix},
-\begin{pmatrix}
	1 & 0 & 0 & 0 & 0 \\
	0 & 1 & 0 & 0 & 0 \\
	0 & 0 & 0 & \zeta^{-1} & 0 \\
	0 & 0 & \zeta & 0 & 0 \\
	0 & 0 & 0 & 0 & 1
\end{pmatrix}.
\]
The model that Burkhardt determined originally \cite{Burkhardt1891}*{p.~208} $y_0^4+8y_0(y_1^3+y_2^3+y_3^3+y_4^3)+48y_1y_2y_3y_4$, arises as the quartic invariant for the transpose action and differs from $B$ by a scaling of $y_0$.

The first two matrices generate the subgroup $\Gamma'$ of matrices defined over $\ZZ[\frac{1}{3}]$. It is isomorphic to $\PGL_2(\FF_3)\rtimes C_2$.

We define the \emph{Hessian} of $B$ to be the projective hypersurface defined by
\[
\He(B)\colon \frac{1}{486}\det \left(\frac{\partial f}{\partial y_i\partial y_j}\right )_{i,j}=0.
\]
The scaling ensures that the resulting polynomial is defined over $\ZZ$ with content $1$. Over any field $k$ of characteristic different from $3$ and containing the cube roots of unity, $\He(B)\cap B$ consists of a union of $40$ planes. Each of these planes contain $9$ of the nodes of $B$. These planes are classically known as \emph{Jacobi-planes}, or $j$-planes.
Furthermore, there are $40$ hyperplanes, classically known as \emph{Steiner primes}, that intersect $B$ in the union of four $j$-planes. Conversely, every $j$-plane lies in four Steiner primes. Two $j$-planes that do not lie in a common Steiner prime are \emph{skew} and meet in a single point, which is a node of $B$.

Over fields not containing a primitive cube root of unity, the intersection $\He(B)\cap B$ splits in eight $j$-planes defined over $k$, four unions of two conjugate $j$-planes meeting in a line, and $12$ unions of two conjugate $j$-planes meeting in a point.

The $j$-planes defined over $k$ are
\[J_i=\{y_0=y_i=0\} \text{ for } i=1,\ldots,4\]
contained in the Steiner prime $y_0=0$ and
\[J_i'=\{y_0+\cdots+y_4=y_0+y_i=0\} \text{ for } i=1,\ldots,4\]
contained in the Steiner prime $y_0+\cdots+y_4=0$. The group $\Gamma'$ acts faithfully on the $J_i'$ and $J'_i$ by (simultaneous) permutation and interchanging $J_i$ with $J'_i$.

Given $\alpha=(\alpha_0:\cdots:\alpha_4)\in B$, one can consider the \emph{polars} (see \cite{Dolgachev12}) of $B$ at $\alpha$. These are hypersurfaces of degrees $3,2,1$ given by
\begin{equation}\label{E:polars}
\begin{aligned}
P_\alpha^{(1)}&=(4y_0^3 + y_1^3 + y_2^3 + y_3^3 + y_4^3)\alpha_0 + 
(3y_0y_1^2 + 3y_2y_3y_4)\alpha_1 + (3y_0y_2^2 +
3y_1y_3y_4)\alpha_2\\
&\quad + (3y_0y_3^2 + 
3y_1y_2y_4)\alpha_3 + (3y_0y_4^2 + 
3y_1y_2y_3)\alpha_4,\\
P_\alpha^{(2)}&=2\alpha_0^2y_0^2 + \alpha_1^2y_0y_1 + \alpha_2^2y_0y_2 + \alpha_3^2y_0y_3 + 
\alpha_4^2y_0y_4 + \alpha_0\alpha_1y_1^2 + \alpha_3\alpha_4y_1y_2 + 
\alpha_2\alpha_4y_1y_3\\
&\quad + \alpha_2\alpha_3y_1y_4 + \alpha_0\alpha_2y_2^2 + 
\alpha_1\alpha_4y_2y_3 + \alpha_1\alpha_3y_2y_4 + \alpha_0\alpha_3y_3^2 + 
\alpha_1\alpha_2y_3y_4 + \alpha_0\alpha_4y_4^2,\\
P_\alpha^{(3)}&=(4\alpha_0^3 + \alpha_1^3 + \alpha_2^3 + \alpha_3^3 + \alpha_4^3)y_0 + 
(3\alpha_0\alpha_1^2 + 3\alpha_2\alpha_3\alpha_4)y_1 + (3\alpha_0\alpha_2^2 +
3\alpha_1\alpha_3\alpha_4)y_2\\
&\quad + (3\alpha_0\alpha_3^2 + 
3\alpha_1\alpha_2\alpha_4)y_3 + (3\alpha_0\alpha_4^2 + 
3\alpha_1\alpha_2\alpha_3)y_4.
\end{aligned}
\end{equation}
One recognizes that $P^{(3)}_\alpha$ is simply the tangent space of $B$ at $\alpha$.

\section{Rational parametrization of the Burkhardt quartic}
\label{S:param}

In this section we give an explicit birational parametrization of the Burkhardt quartic over any field $k$ of characteristic different from $3$. We present our computations over $\QQ$ and observe that the formulas we obtain are defined over $\ZZ$ and maintain their desired properties when reduced modulo a prime different from $3$.

Baker \cite{Baker46} provides an explicit parametrization of $B$ over $\QQ(\zeta_3)$. 
His construction boils down to the observation that given $3$ distinct planes $J_1,J_2,J_3\subset \PP^4$, the variety $L_{J_1,J_2,J_3}$ of lines incident with all of these planes is generally rational of dimension $3$. Furthermore, since $B$ is a hypersurface of degree $4$, a line in $\PP^4$ generally intersects $B$ in $4$ points. If we choose $J_1,J_2,J_3\subset B$, then a line $l\in L_{J_1,J_2,J_3}$ has $3$ of its intersection points with $B$ prescribed by its intersections with $J_1,J_2,J_3$. We obtain a rational map $L_{J_1,J_2,J_3}\dashrightarrow B$ by sending a line to the fourth point of intersection. 

This construction can degenerate in various ways. We are only interested in the component of $L_{J_1,J_2,J_3}$ that parametrize lines that intersect $J_1,J_2,J_3$ in distinct points, since otherwise the map to $B$ is not well-defined. This means that a necessary condition for obtaining a dominant map is that the planes are pairwise skew.

The action of $\PSp_4(\FF_3)$ splits the collection of $\binom{40}{3}$ triples of $j$-planes into $5$ orbits. Only two of these orbits consist of pairwise skew triples and only one of them yields a dominant map. For completeness, we describe all $5$ orbits.
\begin{itemize}
\item $4\cdot40$ triples consisting of planes lying in a single Steiner prime $P$. Any pair of these planes meet in a line. 
\item $2160$ triples consisting of one skew pair, with a third $j$-plane meeting each of the first two in a line.
\item $4320$ triples consisting of a pair of $j$-planes that meet in a line together with a third $j$-plane that is skew to each of the others. 
\item $2\cdot4\cdot 45$ triples of planes that are pairwise skew, but all meet at the same node. 
\item $2880$ triples consisting of mutually skew $j$-planes, pairwise meeting in distinct nodes.
\end{itemize}
The orbit of length $360$ is interesting in its degeneracy.
This configuration arises from the fact that each of the $45$ nodes has $8$ $j$-planes through it, split in two quadruples of pairwise skew planes. Computation shows that any line through $3$ planes in such a quadruple also goes through the fourth. Hence, the resulting map $L_{J_1,J_2,J_3}\dashrightarrow B$ is not dominant.

Baker produces an explicit parametrization, but starts from a configuration that is only defined over $\QQ(\zeta_3)$, not over $\QQ$. Indeed, it is straightforward to check that there is no triple of pairwise skew planes with each plane defined over $\QQ$. We can take two conjugate $j$-planes that are skew and take a third $j$-plane over $\QQ$ that is also skew as follows:
\[
\begin{aligned}
J_1\colon&y_0+\zeta y_1=y_2+\zeta y_3+\zeta y_4=0,\\
J_2\colon&\zeta y_0+ y_1=\zeta y_2+y_3+y_4=0,\\
J_3\colon&y_0=y_3=0.
\end{aligned}
\]
\begin{remark}
Representatives of the other orbits are also straightforward to give: the triple $z_0=z_1=0,z_0=z_2=0,z_0=z_3=0$ represents the orbit of length 160, the triple $J_1,J_2,z_0=z_1=0$ represents the orbit of length 2160, and the triple $J_1,J_2,z_0=z_2=0$ represents the orbit of length 360.
The orbit of length 4320 is represented by the triple $z_0=z_2=0,z_0=z_3=0,z_0+z_1=z_2+z_3+z_4=0$.

In particular, we see that every orbit can be represented by a Galois-stable triple.
\end{remark}

We parametrize an affine patch of $L_{J_1,J_2,J_3}$ by taking, given a point $(t_1,t_2,t_3)$, the line through
\[\begin{aligned}
P=P(t_1,t_2,t_3)&=(1:0:t_2:t_3-t_1:-t_3),\\
Q=Q(t_1,t_2,t_3)&=(0:1:t_1:0:t_1+t_2).\\
\end{aligned}\]
It is clear that $Q$ lies on $J_3$ and that $\zeta P-Q$ and $P-\zeta Q$ lie on $J_1,J_2$ respectively. The fourth linear combination of $P,Q$ that lies on $B$ yields a point $(y_0:y_1:y_2:y_3:y_4)\in B(k(t_1,t_2,t_3))$ and we obtain the following.
\begin{theorem}\label{T:ratpar_full}
Let $k$ be a field of characteristic different from $3$.
The map $\phi\colon \PP^3\dashrightarrow B$ given by the affine chart $(1:t_1:t_2:t_3)\mapsto(\ty_0:\ty_1:\ty_2:\ty_3:\ty_4)$ with
\[
\begin{aligned}
\ty_0&=    t_1^3 - 3t_1^2t_3 - 3t_1t_2^2 - 3t_1t_2t_3 - t_2^3 - 1,\\
\ty_1&=    -t_1^3 + 3t_1^2t_3 - 3t_1t_3^2 + t_2^3 + 1,\\
\ty_2&=    -t_1^4 + t_1^3t_2 + 3t_1^3t_3 - 3t_1^2t_2t_3 - 3t_1^2t_3^2 - 2t_1t_2^3 - 3t_1t_2^2t_3 + t_1 - t_2^4 - t_2,\\
\ty_3&=    -t_1^4 + 4t_1^3t_3 + 3t_1^2t_2^2 + 3t_1^2t_2t_3 - 3t_1^2t_3^2 + t_1t_2^3 - 3t_1t_2^2t_3 - 3t_1t_2t_3^2 + t_1 - t_2^3t_3 - t_3,\\
\ty_4&=    -t_1^4 - t_1^3t_2 + 2t_1^3t_3 + 3t_1^2t_2t_3 + t_1t_2^3 + 3t_1t_2^2t_3 + t_1 + t_2^4 + t_2^3t_3 + t_2 + t_3
\end{aligned}
\]
has birational inverse $\psi\colon B\dashrightarrow\PP^3$ as given in Theorem~\ref{T:ratpar}.
\end{theorem}
\begin{proof}
It is straightforward to check that $\psi\circ\phi$ defines the identity map on an open part. Indeed, we can check this over $\ZZ[\frac{1}{3}]$. This implies that the image of $\phi$ must be $3$-dimensional. By construction, the image of $\phi$ is contained in $B$. Irreducibility of $B$ completes the proof.

The determination of the expressions for $\psi$ is not quite as straightforward. We construct the affine ideal: $(\ty_0 y_i-\ty_i: i = 1,\ldots,4)+(1+y_1^3+y_2^3+y_3^3+y_4^3+3y_1y_2y_3y_4))$ and compute a Gr\"obner basis with respect to an elimination order for $t_1,t_2,t_3$. We then select the basis elements in which the $t_i$ occurs linearly and solve $t_1,t_2,t_3$ from these as rational expressions in $y_1,\ldots,y_4$. This procedure is implemented as \texttt{IsInvertible} by the first author in Magma \cite{magma}.
\end{proof}

\begin{remark}\label{R:parmsym}
As is well known, Baker's parametrization, and hence also the one presented here, is not particularly compatible with the symmetries of $B$. In fact, just a cyclic subgroup of order $9$ pulls back to linear transformations on $\PP^3$. One can determine this by, for instance, determining the $j$-planes that are birational to planes under $\phi,\psi$ (there are $13$) and taking the transformations on $B$ that stabilize this collection. This way we obtain the subgroup generated by the matrix
\[\frac{1}{3}\begin{pmatrix}
-2\zeta - 1 & 2\zeta + 4 & 0 & 0 & 0 \\
\zeta + 2 & -\zeta + 1 & 0 & 0 & 0 \\
0 & 0 & \zeta + 2 & \zeta + 2 & -2\zeta - 1 \\
0 & 0 & \zeta + 2 & \zeta - 1 & \zeta - 1 \\
0 & 0 & -2\zeta - 1 & \zeta - 1 & -2\zeta - 1
\end{pmatrix},\]
inducing the transformation $(t_0:t_1:t_2:t_3)\mapsto
(-3t_0:
(\zeta - 1)t_2 + (-\zeta - 2)t_3:
3\zeta t_1 + (\zeta + 2)t_2 + (-\zeta + 1)t_3:
(-3\zeta - 3)t_1 + (-2\zeta - 1)t_2 + (-\zeta - 2)t_3)$.
\end{remark}

We now proceed with determining the \emph{base locus} of each of the maps $\phi$ and $\psi$. This is the smallest locus of the domain such that the map can be extended to a morphism on the complement.

The base locus of the map $\phi$ has a particular geometric configuration, as described in detail by Finkelnberg \cite{Finkelnberg89}. Over $\QQ(\zeta_3)$ it consists of $9$ lines $l_1,\ldots,l_9$ with $l_i$ meeting $l_{i+1}$ in a point $p_i$, and $l_9$ meeting $l_1$ in $p_9$. The points $\{p_1,p_4,p_7\}$, $\{p_2,p_5,p_8\}$ and $\{p_3,p_6,p_9\}$  define lines that intersect in a common point $p_{10}$ and $l_1\cap l_4\cap l_7$, $l_2\cap l_5\cap l_8$, $l_3\cap l_6\cap l_9$ define a further $3$ points. Finkelnberg proves that any two such configurations in $\PP^3$ are projectively equivalent, and that such a configuration defines the linear system on $\PP^3$ that gives $\phi$. Indeed, an alternative construction of $\phi$ over $\QQ$ is to construct a $\Gal(\Qbar/\QQ)$-invariant configuration like this in $\PP^3$ and prove that the image is isomorphic to $B$ (see \cite{Nasserden16}). 

The map $\psi$ can be defined on a larger part than what is given in Theorem~\ref{T:ratpar}. We compute alternative representations of the map using the following procedure. For a general rational map $\phi\colon X\to Y$ between affine varieties we proceed in the following way. We construct the graph ideal
\[\Gamma=(b_i(x) y_i - a_i(x): i=1,\ldots n)+I(X)+I(Y).\]
We saturate this ideal with respect to $(\prod_{i=1}^n b_i(x))$ and look at the Gr\"obner basis of the resulting ideal with respect to an elimination order on the $y_i$.
We can then select the basis elements in which the $y_i$ appear linearly, and use those relations to find alternative expressions for $y_i$ as rational functions in the $x_j$. For projective varieties, we patch together the affine descriptions.
This procedure is implemented as \texttt{Extend} by the first author in Magma \cite{magma}.

We can apply it to $\psi$ to find, among others, extra representations $(t_0^{(i)}:\cdots:t_3^{(i)})$ with $i=2,3,4$, as given in Appendix~\ref{Appendix} and \cite{BruinNasserden17e},
which together prove that the base locus of $\psi$ is supported on $24$ of the nodes of $B$ (4 defined over $\QQ$ and $10$ quadratic conjugate pairs).

With these explicit descriptions of the birational maps $\phi$ and $\psi$ we can also compute explicit closed subsets $J_\phi$ and $J_\psi$ such that $\phi$ restricts to an isomorphism $\PP^3\setminus J_\phi\to B\setminus J_\psi$. We take them to be the loci where our representations for $\phi$ and $\psi$ are not smooth. We define
\[J_\phi\colon \rk \left(\frac{\partial \xi_i}{\partial t_j}\right)_{i,j}<4,\]
i.e, as the locus of vanishing of the $4\times 4$ minors. We also define 
\[
J_\psi\colon \det\left(
\frac{\partial t^{(i)}_0}{\partial y_j}\;\frac{\partial t^{(i)}_2}{\partial y_j}\;\frac{\partial t^{(i)}_3}{\partial y_j}\;\frac{\partial t^{(i)}_4}{\partial y_j}\;\frac{\partial f}{\partial y_j}\right)_j=0 \text{ for }i=1,2,3,4.
\]
Note that in the latter case we take the locus where \emph{none} of the representatives are smooth.

\begin{remark}\label{R:Jcomponents}
For future reference we record the structure of $J_\phi$ and $J_\psi$.

We can decompose each into irreducible components. We find that $J_\phi$ consists of $3$ plane conics and $18$ lines, all defined over $\QQ(\zeta_3)$.  Two conics are conjugate over $\QQ$ and meet in $3$ points, $4$ pairs of skew lines are conjugate and $5$ pairs of lines meet in a point. The remaining one conic and three lines are defined over $\QQ$.

Decomposition of $J_\psi$ shows that it consists of $15$ $j$-planes defined over $\QQ(\zeta_3)$.
Five pairs of planes are conjugate meeting in a point, one pair meets in a line and one plane is defined over $\QQ$.

We can of course also compute how the components intersect, and we will use this information in Section~\ref{S:zeta}. The intersection data is too voluminous to reproduce here, however.
\end{remark}

\begin{lemma}\label{L:Bisom}
The birational map $\phi$ defined above restricts to an isomorphism
\[\PP^3\setminus J_\phi \to B\setminus J_\psi.\]
\end{lemma}
\begin{proof}
It is certainly the case that $\phi$ induces an isomorphism between $\PP^3\setminus J_\phi$ and its image in $B$. Similarly $\psi$ induces an isomorphism between $B\setminus J_\psi$ and its image in $\PP^3$.

We can check by direct computation that all the components of $J_\phi$ are either part of the base locus of $\phi$ or map into the base locus of $\psi$. In fact, the whole candidate base locus of $\psi$ gets hit, so we verify in the process that we really have found the base locus of $\psi$.

Similarly, all the components of $J_\psi$ map into the base locus of $\phi$.
It follows that $J_\phi$ and $J_\psi$ are \emph{minimal} so that $\PP^3\setminus J_\phi$ and $B\setminus J_\psi$ are isomorphic to their images under $\phi$ and $\psi$ respectively. It follows they must be the images of each other.
\end{proof}

\section{The zeta function over $\FF_q$}
\label{S:zeta}

\begin{dfn} Let $X$ be an algebraic variety, not necessarily closed, defined over the finite field $\FF_q$. The \emph{zeta function} of $X$ is the formal power series
	\[Z(X/\FF_q,T)=\exp \left(\sum_{n=1}^\infty \#X(\FF_{q^n})\frac{T^n}{n}\right)\]
\end{dfn}
Standard properties of zeta functions include
\begin{lemma}\label{L:zetacalc}\leavevmode
	\begin{enumerate}
	\item $\displaystyle Z(\PP^n/\FF_q,T)=\frac{1}{(1-T)\cdots(1-q^nT)}.$
	
	\item Let $X,Y$ be algebraic varieties over $\FF_q$. Then
	\[Z(X\cup Y/\FF_q,T)Z(X\cap Y/\FF_q,T)=Z(X/\FF_q,T)Z(Y/\FF_q,T).\]
	
	\item Suppose that over $\FF_{q^2}$ we have $X=Y\cup Y'$, where $Y,Y'$ are disjoint and conjugate over $\FF_q$. Then
	\[Z(X/\FF_q,T)=Z(Y/\FF_{q^2},T^2).\]
	
	\end{enumerate}
\end{lemma}

Together with Lemma~\ref{L:Bisom} this gives that
\[Z(B/\FF_q,T)=\frac{Z(J_\psi/\FF_q,T)\,Z(\PP^3/\FF_q,T)}{Z(J_\phi/\FF_q,T)},\]
and, since $J_\phi,J_\psi$ are varieties that are unions of varieties that are isomorphic to $\PP^n$ or unions of conjugate varieties, with intersections that are also of this type, we can use Lemma~\ref{L:zetacalc} to compute the right hand side.

In order to compute $Z(J_\phi/\FF_q,T)$ and $Z(J_\psi/\FF_q,T)$ we need to do a careful inclusion-exclusion argument which is too big to do by hand: for $J_\psi$ it involves more than $200$ components. We sketch a formal description that is suitable for implementation in a computer algebra system.

Suppose $\{X_1,\ldots,X_m\}$ is a collection of algebraic varieties over $\FF_q$ that is closed under taking intersections. Define
\[M_{ij}=\begin{cases}
1&\text{ if }X_j\subseteq X_i\\
0&\text{ otherwise}
\end{cases}\]
Solve over $\ZZ$ the linear equation
\[
(e_1,\ldots,e_m)M=(1,\ldots,1).
\]
Then
\[Z(\bigcup_{i=1}^m X_i/\FF_q,T)=\prod_{i=1}^m Z(X_i/\FF_q,T)^{e_i}\]

With Remark~\ref{R:Jcomponents} we see that the observations in Lemma~\ref{L:zetacalc} allow us to compute the zeta functions of the components and their intersections, if we note that over $\FF_q$, a nonsingular conic is isomorphic to $\PP^1$ and that the zeta function of two conjugate intersecting lines can be computed as $\frac{1}{(1-T^2)(1-qT^2)}\frac{1-T^2}{1-T}$, and similarly for two conjugate planes meeting in a line or a point.

\begin{proof}[Proof of Theorem~\ref{T:zetaB}]
Combining Lemmas~\ref{L:Bisom} and \ref{L:zetacalc}, we obtain
\[Z(B/\FF_q,T)=\frac{Z(\PP^3,T)Z(J_\psi,T)}{Z(J_\phi,T)}.\]
Furthermore, for both $J_\phi$ and $J_\psi$ we have a decomposition into varieties for which Lemma~\ref{L:zetacalc}(1, 3) gives us the zeta functions. This gives us the required formula.
\end{proof}

\begin{proof}[Proof of Corollary~\ref{C:zetaBdesing}]
When we desingularize $B$ by blowing up the $45$ nodes, we replace each node $\alpha$ by the projection (from $\alpha$) of the tangent cone at $\alpha$. If $q\equiv 1\pmod 3$, then all the nodes are defined over $\FF_q$, and each get replaced by a quadric with a split system of lines, i.e., $\PP^1\times\PP^1$. For the zeta function this gives a correction factor of $(1-qT)^{-2}(1-q^2T)^{-1}$ for each.

If $q\not\equiv 1\pmod 3$ then there are $19$ pairs of conjugate nodes whose tangent cones are split over their fields of definition, $6$ nodes with a split tangent cone, and one node $(-1:1:1:1:1)$ which has a non-split tangent cone. For this the correction factor is $((1+qT)(1-qT)(1-q^2T))^{-1}$.
\end{proof}

\begin{proof}[Proof of Corollary~\ref{C:pointcount}]
We can also determine $Z(B\cap\He(B)/\FF_q,T)$ via the same procedure. Using that
\[\#X(\FF_q)=- \left.\frac{Z(X/\FF_q,T)}{\frac{d}{dT}Z(X/\FF_q,T)}\right|_{T=0}\]
we get the formulas as stated. Note that for $q\equiv 1\pmod{3}$ the given formula already follows from \cite{HoffmanWeintraub01}.
\end{proof}

\begin{remark}\label{R:zeta_small_q}
We see that for $q=4,7,13$, all rational points on $B$ lie on $j$-planes. For those $q$, there are no genus $2$ curves over $\FF_q$ with a Jacobian that has fully rational $3$ torsion. In fact, as Noam Elkies pointed out in a private conversation, for $q=16,19$, the number of rational points outside the $j$-planes is a divisor of the order of the Burkhardt group. Indeed, for those $q$, the rational points outside the $j$-planes form a single orbit, so for each there is a unique isomorphism class of genus $2$ curves with fully rational $3$-torsion. For $q=16$, this class is represented by the quadratic twist of the affine model $y^2+y=x^5$
and for $q=19$ by $y^2=x^6+8x^3+1$.
\end{remark}

\section{Models of genus $2$ curves}
\label{S:KummerWeddle}
A nonsingular curve of genus $2$ is hyperelliptic. It can be represented as a separable double cover of $\PP^1$, ramified over a degree $6$ locus. Over fields $k$ of characteristic not equal to $2$, it admits a weighted projective model
\[C\colon y^2=f(x,z)=f_0x^6+f_1x^5z+\cdots+f_6z^6,\]
where $x,y,z$ have weights $1,3,1$ respectively.
The \emph{quadratic twist} of $C$ by $\sqrt{d}$ is given by a model
\[C^{(d)}\colon y^2=d\,f(x,z).\]
It is isomorphic to $C$ over $k(\sqrt{d})$. It follows that by marking $6$ points on a $\PP^1$ one specifies a genus $2$ curve up to quadratic twists.

We write $\sJ_C$ for the \emph{Jacobian variety} of $C$, which is a principally polarized Abelian surface representing $\Pic^0$, and we write $\sK_C=\sJ_C/\langle -1\rangle$ for the associated \emph{Kummer surface}. The surface $\sK_C$ admits a quartic model in $\PP^3$, with $16$ nodal singularities (the image of $\sJ_C[2]$). It follows that $\sK_C$ comes with one marked node: the image of the origin on $\sJ_C$.

There is also a surface $\sP_C$ that represents $\Pic^1$. It is a principal homogeneous space under $\sJ_C$. There is a natural embedding $C\to\sP_C$, sending a point to its divisor class.

The projective dual of $\sK_C$, denoted by $\sK_C^*$ is also a quartic surface with $16$ nodes, in the dual space $(\PP^3)^*$. If $f(x,z)$ has a $k$-rational root then $\sK_C^*$ and $\sK_C$ are isomorphic over $k$. In general this is not the case, however. The variety $\sP_C$ defined above has an involution $\iota$ induced by the hyperelliptic involution on $C$, and $\sP_C/\langle\iota\rangle$ is isomorphic to $\sK_C^*$ (see \cite{Cassels-Flynn96}*{Ch.~4}).

The $16$ nodes on $\sK_C$ correspond to $16$ \emph{tropes} on $\sK_C^*$: these are planes that contain $6$ nodes. They intersect $\sK_C^*$ in a double-counting conic.
Since we need to distinguish here between several kinds of Kummer surfaces that geometrically are all the same, we introduce some terminology.

By a \emph{geometric Kummer surface} we mean a quartic surface in $\PP^3$ with $16$ nodal singularities. A \emph{Kummer surface} is one with a marked node over $k$. A \emph{dual Kummer surface} is a geometric Kummer surface with a marked trope over $k$.

If a dual Kummer surface indeed comes from a curve $C$ over $k$, then the conic on the marked trope is isomorphic to $\PP^1$ and the $6$ nodes on it mark $C$ up to quadratic twist. As is well-known, there is a field-of-moduli versus field-of-definition obstruction for curves of genus $2$ and dual Kummer surfaces on which the conic is not isomorphic to $\PP^1$ do exist over non-algebraically closed base fields.

The most straightforward way to show that a conic is isomorphic to $\PP^1$ is to exhibit a rational point on it. However, in our application this is a slightly unnatural criterion: the fact that a conic is isomorphic to $\PP^1$ does not mark any particular point on the conic.

There is an alternative description, exploiting a phenomenon known as the \emph{association} of point sets \cite{Coble1922}.  For us it yields that the moduli of $6$ points in $\PP^1$ and of $6$ points in $\PP^3$ are essentially equivalent (see for instance \cite{Howard-ea08}): if one maps $\PP^1$ into $\PP^3$ (with coordinates $(x_0:x_1:x_2:x_3)$ via a complete linear system of degree $3$, the $6$ points end up in general position (meaning, no $3$ in a line, no $4$ in a plane). Conversely, $6$ points in general position in $\PP^3$ determine a rational normal curve of degree $3$. They also determine a $4$-dimensional system of quadrics $Q=\langle Q_1,\ldots,Q_4\rangle$ having these $6$ points as a base locus. 
We follow a classical construction (see \cite{Baker1907}*{III.17} and \cite{Coble1929}*{III.41}; also in \cite{Cassels-Flynn96}*{Chapter~5}). If $f(x,z)=f_0x^6+f_1x^5z+\cdots+f_6z^6$, we can take the standard Veronese embedding $(x:z)\mapsto(x^3:x^2z:xz^2:z^3)$ and obtain
\begin{equation} \label{eq:standardquadsinP3}
\begin{split} Q_{1}&=x_{0}x_{2}-x_{1}^{2},\\
Q_{2}& =x_{0}x_{3}-x_{1}x_{2},\\
Q_{3}& =x_{1}x_{3}-x_{2}^{2},\\
Q_{4}&=f_0x_0^2+f_1x_0x_1+f_2x_1^2+f_3x_1x_2+f_4x_2^2+f_5x_2x_3+f_6x_3^2,\end{split}  
\end{equation}
where
\[L\colon Q_1=Q_2=Q_3=0\]
is the image of the Veronese embedding.

Such a system $Q$ gives rise to $2$ birational quartic surfaces (see \cite{Cassels-Flynn96}*{Ch.~5}). First we consider the locus of points in $\PP^3$ that occur as singularities of singular members of $Q$:
\[\sW_Q\colon \det\left(\frac{\partial Q_i}{\partial x_j}\right)_{i,j}=0,\]
which is classically known as a \emph{Weddle surface}.
The singular members themselves are given by
\[\sK_Q^*\colon \det(\eta_1Q_1+\eta_2Q_2+\eta_3Q_3+\eta_4Q_4)=0,\]
which is classically known as the \emph{symmetroid} of $\sW_Q$. The birational map $\sW_Q\dashrightarrow \sK_Q^*$ is induced by the relation, given $P\in\sW_\sQ$:
\[\left.\frac{\partial}{\partial x_j}(\eta_1Q_1+\cdots +\eta_4Q_4)\right|_P=0\text{ for }j=0,1,2,3\]

It is classical that $\sK_Q^*$ is a geometric Kummer surface, and that $\eta_4=0$ is a trope, so it is a \emph{dual} Kummer surface. We write $(\PP^3)^*$ for the ambient space and write $\PP^3$ for its dual, with coordinates $(\xi_1:\xi_2:\xi_3:\xi_4)$ dual to $(\eta_1:\eta_2:\eta_3:\eta_4)$. Then the dual of $\sK_Q^*$ is exactly the model of $\sK_C$ as given in \cite{Cassels-Flynn96}*{Ch.~5.5}, so we have $\sK_Q^*=\sK_C^*$. 

The composition of $\PP^1\to\sW_Q\to \sK^*_Q$ is given by $(x:z)\mapsto(z^2:-xz:x^2:0)$ and the image is $\eta_4=(\eta_1\eta_3-\eta_2^2)=0$. 

We write $(\PP^2)^*\subset (\PP^3)^*$ for the plane $\eta_4=0$. Under duality this corresponds to the projection $\PP^3\to\PP^2$ given by $(\xi_1:\cdots\xi_4)\mapsto(\xi_1:\xi_2:\xi_3)$, giving a natural duality between $\PP^2$ and $(\PP^2)^*$: A point $(\xi_1:\cdots:\xi_4)\in\PP^3$ determines a plane $\xi_1\eta_1+\cdots+\xi_4\eta_4=0$ in $(\PP^3)^*$, which when intersected with $\eta_4=0$ determines a line  $\xi_1\eta_1+\cdots+\xi_3\eta_3=0$ in $(\PP^2)^*$.

This explicit duality gives us a coordinate-free way to express for a point $D\in \sJ_C$ the relation between the image on $\sK_C$ and the support on $\PP^1$ of a representing effective divisor of degree $2$.

\begin{proof}[Proof of Proposition~\ref{P:kummer_find_intersection}]
We can check this over a field extension where $D$ is supported on degree $1$ points. First suppose $x_*(D)$ is separated.
Let $D=(x_1:y_1:z_1)+(x_2:y_2:z_2)$. Then the image on $L$ is $(z_1^2:-x_1z_1:z_1^2:0)+(z_2^2:-x_2z_2:z_2^2:0)$. The line through these points is described by $\eta_4=\xi_1\eta_1+\xi_2\eta_2+\xi_3\eta_3=0$, where $(\xi_1:\xi_2:\xi_3)=(z_1z_2:x_1z_2+x_2z_1:x_1x_2)$. 
But with the coordinates for $\sK_C$ used in \cite{Cassels-Flynn96} that is the image of $D-\kappa_C$ under $\sK_C\to\PP^2$, so we see that the tangent plane $\xi(D)$ indeed intersects the trope $\eta_4=0$ in the line that intersects $L$ in $x_*(D)$.

If $x_*(D)$ is not separated it is straightforward to check that $\eta(D)$ is tangent to $L$.
\end{proof}

\section{Explicit moduli interpretation of the Burkhardt quartic}
\label{S:moduli}

\subsection{Level structure}

\begin{dfn}\label{D:level3structure} A \emph{full level-$3$ structure} for us will be a group scheme $\Sigma=\Sigma_{g,3}$ over $k$ that over $k^\sep$ is isomorphic to $(\ZZ/3)^{2g}$, and is equipped with a non-degenerate alternating pairing $\Sigma\times\Sigma\to\mu_3$.
An Abelian surface with full level-$3$ structure $\Sigma_{2,3}$ is a principally polarized Abelian surface $A$ with an embedding $\Sigma_{2,3}\to A$ such that the pairing on $\Sigma_{2,3}$ is compatible with the Weil pairing on $A[3]$.
\end{dfn}

One full level-$3$ structure is $\Sigma=(\ZZ/3)^2\times (\mu_3)^2$, where the pairing comes from considering $(\mu_3)^2$ as the Cartier dual of $(\ZZ/3)^2$.
It is known \cite{FreitagSalvati04} that the normalization of the projective dual of the Burkhardt quartic is isomorphic to the Satake compactification of the moduli space $\sA_{2,3}$ of Abelian surfaces with full level-$3$ structure $\Sigma$.

The open part $B\setminus\He(B)$ (which is nonsingular, and hence isomorphic to a part of the dual) is the part corresponding to Jacobians of smooth genus $2$ curves.
Since $\sA_{2,3}$ is a fine moduli space,
it follows that there is a \emph{universal} genus $2$ curve $C_\alpha$ over
$B\setminus\He(B)$, such that if $\alpha$ is a point on $B$ then $\sJ_\alpha=\sJ_{C_\alpha}$ is the corresponding Abelian variety with level structure. We write $\sK_\alpha$ and $\sK_\alpha^*$ for the Kummer surface and its dual, respectively. We will explicitly construct a model of the curve $C_\alpha$ and the data that marks the level-$3$ structure on it.

\subsection{Explicitly marking a level structure}

\begin{dfn}
Given a group scheme $\Sigma$ over $k$ and a quadratic extension $k(\sqrt{d})/k$, we define the \emph{quadratic twist} $\Sigma^{(d)}$ by the short exact sequence
\[0\to\Sigma^{(d)}\to\Res_{k(\sqrt{d})/k}\Sigma\to \Sigma\to 0,\]
where $\Res_{k(\sqrt{d})/k}(\Sigma)$ stands for the Weil restriction of scalars and the third arrow is the map induced by the norm from $k(\sqrt{d})$ to $k$.
\end{dfn}
In particular, we note that $\mu_3=(\ZZ/3)^{(-3)}$.

\begin{prop}\label{P:cyclic} Let $C\colon y^2=F(x,z)$ be a model of a genus $2$ curve over a field $k$ of characteristic distinct from $2$, where $\deg(F)=6$ Then the cyclic order $3$ subgroups of $\Pic^0(C/k^\sep)$ isomorphic to $(\ZZ/3\ZZ)^{(d)}$ as a Galois module are in bijection with decompositions of the form
\[F(x,z)=d(G(x,z)^2+4\lambda H(x,z)^3),\]
where $\deg(H)=2$ and $\deg(G)=3$.
\end{prop}
\begin{proof}
First assume we have a decomposition of the required form.
The effective degree $2$ divisors $D_1=\{H(x,z)=0,y-\sqrt{d}G(x,z)$ and $D_2=\{H(x,z)=0,y+\sqrt{d}G(x,z)\}$ are defined over $k$ if $d$ is a square and quadratic conjugate over $k(\sqrt{d})$ otherwise. We write $\kappa=\{z=0\}$ for the effective canonical divisor supported at $z=0$. Then $D_1+D_2$ is linearly equivalent to $2\kappa$ and the divisor of the function $(y-G)/z^3$ is $3(D_1-\kappa)$. This shows $\{0,[D_1-\kappa],[D_2-\kappa]\}\subset \Pic^0(C/k^\sep)$ is a subgroup with the required Galois module structure. This representation is unique because for any non-zero divisor class there is a unique effective degree $2$ divisor $D$ such that $D-\kappa$ represents the class.

Conversely, given a divisor $D-\kappa$, where the direct image of $D$ on the $\PP^1$ with coordinates $(x:z)$ is determined by $H(x,z)=0$, it is straightforward to check that if $3(D-\kappa)$ is principal, the function bearing witness to that fact gives rise to $\lambda, G$. 
\end{proof}

\begin{cor}\label{C:3level_marking}
Let $C\colon y^2=F(x,z)$ be a genus $2$ curve over a field $k$ of characteristic different from $2$ in which $-3$ is not a square.
A full level-$3$ structure (up to conjugation on $\mu_3$) on $\sJ_C$ is given by $4$ distinct decompositions
\[\begin{split}
F&=G_1^2+4\lambda_1H_1^3=G_2^2+4\lambda_2H_2^3\\
F&=-3((G_1')^2+4\lambda_1'(H_1')^3)=-3((G_2')^2+4\lambda_2'(H_2')^3)
\end{split}\]
\end{cor}
\begin{proof}
By Proposition~\ref{P:cyclic} the decompositions mark $(\ZZ/3\ZZ)^2\subset\sJ_C[3]$ and $(\mu_3)^2\subset\sJ_C[3]$, so it follows that $\sJ_C\simeq (\ZZ/3\ZZ)^2\times(\mu_3)^2$. 
The Weil pairing necessarily restricts to the trivial pairing on $(\ZZ/3\ZZ)^2$ and its nondegeneracy induces a natural identification on $(\mu_3)^2$ with the Cartier dual of $(\ZZ/3\ZZ)^2$. As a result, a basis choice for $(\ZZ/3\ZZ)^2$, which is given by the first two decompositions, also induces a natural basis choice on $(\mu_3)^2$ by taking a dual basis.
\end{proof}

\subsection{Some results by Coble}

Coble \cite{Coble1917}*{(52)} (see Hunt\cite{Hunt96} for a more modern exposition) gives a model for $\sP_\alpha\subset\PP^8$ as an intersection of $9$ quadrics. He works over $k=\CC$, so $\sP_\alpha$ is isomorphic to $\sJ_\alpha$, but not canonically so. Indeed, an origin is not marked on $\sP_\alpha$.

He also gives direct constructions for the Weddle surface $\sW_\alpha$ \cite{Coble1917}*{p.~362 above (70)} and its symmetroid $\sK_\alpha^*$ \cite{Coble1917}*{p.~360 (63)}.
One can recognize from his description a $4$-dimensional system of quadrics through $6$ points spanned by
\begin{equation}\label{eq:The quadrics in z} \begin{array}{lcl} Q_1=\alpha_0z_0^2-\alpha_2z_2z_3-\alpha_3z_1z_3-\alpha_4z_2z_1,\\ Q_2=\alpha_0z_1^2+\alpha_1z_2z_3+\alpha_3z_0z_3-\alpha_4z_2z_0,\\  Q_3=\alpha_0z_2^2+\alpha_1z_1z_3-\alpha_2z_0z_3+\alpha_4z_0z_1,\\ Q_4=\alpha_0z_3^2+\alpha_1z_2z_1+\alpha_2z_0z_2-\alpha_3z_0z_1, \end{array} \end{equation} 
from which one can recover $\sW_\alpha$ and $\sK^*_\alpha$. Coble \cite{Coble1917}*{bottom of p.~364} also gives a direct way of constructing from $\alpha$ a genus $0$ curve with $6$ marked points:
let $\pi_\alpha\colon\PP^4\to\PP^3$ be the projection away from $\alpha$. Then $\pi_\alpha(P_\alpha^{(3)})$ defines a plane, $\pi_\alpha(P_\alpha^{(2)}\cap P_\alpha^{(3)})$ defines a plane conic, and $\pi_\alpha(P_\alpha^{(1)}\cap P_\alpha^{(2)}\cap P_\alpha^{(3)})$ marks $6$ points on that conic.
 Coble considers the \emph{enveloping cone} $\EC_\alpha(P_\alpha^{(1)})$ and proves that $\pi_\alpha(\EC_\alpha(P_\alpha^{(1)}))$ is a model for $\sK_\alpha^*$ and that $\pi_\alpha(P^{(3)}_\alpha)$ marks a trope on it and that the $6$ points marked by the intersection of the polars are indeed nodes of $\sK_\alpha^*$.

Since we have an expression for $\sK_\alpha^*$ as a symmetroid, we can find a parametrization $\Phi_\alpha\colon\PP^1\to \pi_\alpha(P^{(2)}_\alpha\cap P^{(3)}_\alpha)$, and $\Phi_\alpha^{-1}(\bigcap_i P^{(i)}_\alpha)$ gives us $6$ points on a $\PP^1$. This determines $C_\alpha$ up to quadratic twist. We execute this procedure in the following section. The explicit marking of the level-$3$ structure as given in Section~\ref{S:marking} will confirm which twist we should take.

\begin{remark}
In the arithmetic setting the difference between a Kummer surface and its dual is more pronounced, so it is perhaps worthwhile to remind the reader of the construction explained by Coble.

The construction of $\sK_\alpha^*$ arises from the fact that the cubic polar of $B$ at $\alpha$ is isomorphic to Segre's cubic threefold over an algebraic closure of $k$. The projective dual of a Segre cubic is an Igusa quartic, so from a point $\alpha$ on $B$ we obtain a point $\alpha^*$ on a twist of an Igusa quartic, corresponding to the tangent space of $B$ at $\alpha$.

It is classical (see \cite{Dolgachev12} or \cite{Hunt96} for an account and references) that the Igusa quartic has an interpretation as the moduli space of Kummer surfaces with full level-2 structure (which consists of a labeling of the nodes). This interpretation is realized by intersecting the Igusa quartic with the tangent space at a point on it. The point itself marks one node and the components of the singular locus mark the remaining $15$. Under projective duality one can check that this intersection corresponds to the enveloping cone at $\alpha$, leading to the construction of $\sK_\alpha^*$ sketched above.

Hence we see that in fact there is a very direct way to construct from a point $\alpha\in B$ the corresponding twist of the moduli space of Kummer surfaces (and dual Kummer surfaces) with full level-$2$ structure. For further details of this surprising fact we refer the reader to Coble and Hunt.
\end{remark}

\subsection{Explicit construction of $\Phi_\alpha$}\label{S:explicit_curve}

We fix coordinates $(\eta_1:\cdots:\eta_4)$ on the codomain of $\pi_\alpha$ by identifying it with the hyperplane $y_0=0$, so that $\pi_\alpha(0:y_1:\cdots:y_4)=(y_1:\cdots:y_4)$.
We find that $\sK^*_\alpha=\pi_\alpha(\EC_\alpha(P_\alpha^{(1)}))$ has equation $\det(\sum_{i=1}^4 \eta_iQ_i)=0$ and that the trope $\pi_\alpha(P_\alpha^{(3)})$ has equation
\[(\alpha_0\alpha_1^2+\alpha_2\alpha_3\alpha_4)\eta_1+
(\alpha_0\alpha_2^2+\alpha_1\alpha_3\alpha_4)\eta_2+
(\alpha_0\alpha_3^2+\alpha_1\alpha_2\alpha_4)\eta_3+
(\alpha_0\alpha_4^2+\alpha_1\alpha_2\alpha_3)\eta_4=0.\]
Let $\langle Q_1',Q_2',Q_3'\rangle\subset \langle Q_1,Q_2,Q_3,Q_4\rangle$ be the subspace corresponding to this plane. It defines a space cubic $L_\alpha$ on $\sW_\alpha$.

The line through $(1:0:0:0)$ and $(0:\alpha_2:\alpha_3:\alpha_4)$ intersects $L_\alpha$ in $2$ points, so the planes through this line intersect $L_\alpha$ in a third point. We parametrize these planes using $(x:z)$ by the row span of
\[\begin{pmatrix}
1&0&0&0\\
0&\alpha_2&\alpha_3&\alpha_4\\
0&x&z&0
\end{pmatrix}
\]
By restricting $Q_1',Q_2'$ to this space, we can determine the third intersection point on $L_\alpha$ using $(\eta_1:\cdots:\eta_4)$ as functions in $(x:z)$. We can then find $f_\alpha(x,z)$ (up to scaling) by solving for $Q_1=Q_2=Q_3=Q_4=0$.
For $\alpha=(1:\alpha_1:\alpha_2:\alpha_3:\alpha_4)$ we find $f_\alpha(x)=G_3(x)^2+\lambda_3H_3(x)^3$ as in Proposition~\ref{P:explicit_curve}.

\begin{proof}[Proof of Proposition~\ref{P:explicit_curve}]
The argument above indicates that $y^2=f_\alpha(x)$ should be a model of $C_\alpha$ (up to quadratic twist) whenever $f_\alpha$ is square-free. From our choice of coordinates, it is clear this happens if $\alpha\notin\He(B)$ and $\alpha_4\neq 0$ and indeed, factorization of the discriminant of $f_\alpha$ confirms this. Furthermore, over fields where $2$ is invertible, the model is birational to the one given here.

In order to determine the appropriate twist, we observe that if a curve has $(\ZZ/3)^2\times(\mu_3)^2\subset J_\alpha$ then only $C_\alpha$ and its quadratic twist $C^{(-3)}_\alpha$ have $3$-torsion points defined over the base field (and indeed, taking a quadratic twist by $\sqrt{-3}$ yields an isomorphic level structure). As we will see in Section~\ref{S:marking}, the model given shows that $C_\alpha$ does have a $3$-torsion point, which verifies that we have the right twist.
\end{proof}

\begin{remark}
While the theory of Weddle and Kummer surfaces as employed here is not valid in characteristic $2$, \cite{vdGeer87}*{Theorem~3.1} yields that the moduli interpretation of $B$ holds over $\ZZ[1/3]$. In fact, while the models for $C_\alpha$ of the form $y^2=G^2+4\lambda H^3$ have bad reduction at $2$, the model
\[y^2+Gy=\lambda H^3\]
is equivalent if $2$ is invertible, and generally does have good reduction at $2$. 
\end{remark}

\subsection{Marking the $3$-torsion}
\label{S:marking}

Let $J$ be a $j$-plane. Computation shows that $\pi_\alpha(J)$ is a tangent plane to $\sK_\alpha^*$, so it corresponds to a point on $\sK_\alpha$. Using Proposition~\ref{P:kummer_find_intersection} we can find the corresponding degree $2$ divisor on the $(x:z)$-line.

In fact, for $i=1,\ldots,4$, the $j$-plane $J_i=\{y_0=y_i=0\}$ gives rise to a rational point on the Jacobian of the curve $C_\alpha$ as in Proposition~\ref{P:explicit_curve}, and further computation shows that the relevant point is of order $3$ and we obtain a decomposition $F=G_i^2+\lambda_iH_i^3$, as in Corollary~\ref{C:3level_marking}. Indeed, the form given in Proposition~\ref{P:explicit_curve} is the decomposition for $i=3$.

\begin{cor} \leavevmode
\begin{enumerate}
\item[(a)] The $j$-planes are in Galois covariant bijective correspondence with the order $3$ subgroups of $\sJ_\alpha[3]$.

\item[(b)] The marking of the level-$3$ structure on $C_\alpha$ can be described by a plane conic $\pi_\alpha(P^{(2)}_\alpha\cap P^{(3)}_\alpha)$ with $6$ points $\pi_\alpha(P^{(1)}_\alpha\cap P^{(2)}_\alpha\cap P^{(3)}_\alpha)$ together with $40$ lines $\pi_\alpha(J)\cap P^{(3)}$, where $J$ runs through the $j$-planes.

\item[(c)] Two cyclic subgroups pair trivially if and only if the corresponding $j$-planes lie in a common Steiner prime.
\end{enumerate}
\end{cor}
\begin{proof}
For (a) and (b), the computation referenced above shows that a particular $j$-plane marks a degree $2$ effective divisor on the $(x:z)$-line that corresponds to an order $3$ subgroup.
The full result now follows by symmetry, because $\Sp_4(\FF_3)$ acts transitively on the $40$ order $3$ subgroups of $\sJ_\alpha[3]$, as well as on the $j$-planes, and acts via linear transformations on $B\subset\PP^4$.

For (c), note that $J_1,J_2$ lie in the common Steiner prime $\{y_0=0\}$ and that their corresponding $3$-torsion points are defined over the base field, which doesn't necessarily contain a cube root of unity. Since the Weil pairing is Galois covariant, it follows they must pair trivially.  Alternatively, one can check this by explicitly computing the pairing or by using the criterion given in \cite{BruinFlynnTesta14}.

The general result now follows from the fact that under $\Sp_4(\FF_3)$ the order $9$ subgroups form two orbits, one of length $40$ consisting maximal isotropic subspaces, and one of length $90$, consisting of the other spaces. Indeed, there are exactly $40$ Steiner primes.
\end{proof}

In order to show that the level-$3$ structure marked is indeed of the type $\Sigma_{2,3}$ we appeal to Corollary~\ref{C:3level_marking}. We have already seen that $J_1,\ldots,J_4$ give rise to decompositions of the first type. 
The $j$-planes $J'_i=\{y_0+\cdots+y_4=y_0+y_i=0\}$ similarly give rise to decompositions of the form $F=-3((G_i')^2+4\lambda_i'(H_i')^3)$.

\begin{remark} By a curious coincidence the decompositions coming from $J_1,\ldots,J_4$ computed in the way suggested above, hold regardless of the Burkhardt relation, and so does the decomposition specified by $J_4'$. It follows that for any $\alpha\in\PP^4$ outside a certain closed subset, the curve $C_\alpha$ has a Jacobian with a $(\ZZ/3)^2\times\mu_3$-level structure marked on it. We only need $\alpha\in B$ to get the second copy of $\mu_3$. We list the relevant decompositions in Appendix~\ref{Appendix} and \cite{BruinNasserden17e}.
\end{remark}

\subsection{Genus $2$ curves as cubic discriminants}
\label{S:disc}
Genus $2$ curves also arise as discriminants of degree $3$, genus $1$ covers of $\PP^1$. Indeed, such a genus $1$ curve $E$ has a degree $3$ divisor, so $E$ admits a cubic model in $\PP^2$. Over fields of characteristics different from $2,3$ we can assume that $E$ is given by a model
\[E\colon w^3+3Hw+2G=0,\]
where $H,G\in k[x,z]$ are forms of degrees $2,3$ respectively, and our degree $3$ map to $\PP^1$ is given by $(x:z:w)\mapsto(x:z)$.

The discriminant of this cubic with respect to $w$ is $-3\cdot 6^2(G^2+H^3)$, which is square-free precisely if $E$ is nonsingular and $H,G$ are coprime. In that case,
\[C\colon y^2=-3(G^2+H^3)\]
is a genus $2$ curve and $D= C\times_{\PP^1} E$ is an unramified $\mu_3$-cover of $C$ obtained by adjoining a cube root of the function $(G-y)/z^3$.
Indeed, by geometric class field theory, specifying an order $3$ subgroup of $\sJ_C$ amounts to specifying an unramified (geometrically Galois) cover $D\to C$. Furthermore, the involution that generates $\Aut(D/C)$ is the pull-back of the hyperelliptic involution on $C$, so a quadratic twist of $C$ has a corresponding quadratic twist of $D$ as a cover, with the same quotient $E$.
Hence we see that specifying a point $\alpha$ on the Burkhardt together with a $j$-plane amounts to specifying a cubic genus $1$ cover of $\PP^1$.

Baker and Hunt observe that the cubic polar $P^{(1)}_\alpha$ cuts out a Hesse pencil on a given $j$-plane as $\alpha$ varies. We verify that this is indeed the relevant cubic $E_{J,\alpha}$ and identify the relevant $3$-cover.

\begin{proof}[Proof of Proposition~\ref{P:disc}] We set $(y_2:y_3:y_4)=(\alpha_2w:\alpha_3w+x:\alpha_4w+z)$ and take the discriminant of the resulting cubic with respect to $w$. This gives a sextic form in $x,z$. We compute and compare the Igusa invariants of this form and of $C_\alpha$ and find that they agree up to weighted projective equivalence on an open part of $B$. Since Igusa invariants classify sextic forms up to scaling, this verifies that $C_\alpha$ (up to quadratic twist) occurs as the discriminant.

It follows that $C_\alpha\times_{\PP^1} E_{J,\alpha}\to C_\alpha$ is an unramified, geometrically Abelian, cover and hence capitalizes some order $3$ subgroup of $\sJ_\alpha$. We can check computationally which one by specializing $\alpha$ to a point where the triples $(H,\lambda,G)$ associated to the $8$ $j$-planes $J_i,J_i'$ lead to cubics $w^3+3\lambda H w+\lambda G=0$ have pairwise distinct $j$-invariants and check the appropriate identity holds for the particular point. It follows on an open by continuity.
\end{proof}

\begin{remark}\label{R:cubic_cover_coordinate_free}
The map $(\alpha_0:\cdots:\alpha_4)\mapsto (\alpha_2:\alpha_3:\alpha_4)$ that gives the data for the cubic cover $E_{J,\alpha}\to\PP^1$ can be described in the following coordinate-free way. The $j$-plane $J_1$ is contained in four Steiner primes. Each of these Steiner primes contains $3$ other $j$-planes that intersect $J_1$ in a line. Each such triple of $j$-planes intersect in a common point. Hence, a $j$-plane gives rise to $4$ points. These turn out to be collinear. For instance, for $J_1\colon x_0=x_1=0$, these points are $(0:1:0:0:0),(-1:1:0:0:0),(-\zeta_3:1:0:0:0),(\zeta_3+1:1:0:0:0)$. For $J_1$ we find the line is $L_{J_1}\colon x_2=x_3=x_4=0$. The point $\alpha$ gets mapped to $J_1$ by taking the plane spanned by $\alpha$ and $L_{J_1}$ and intersecting it with $J_1$.
\end{remark}

\section*{Acknowledgments}
We thank Riccardo Salvati Manni for pointing out a subtlety in the relation between the dual of the Burkhardt quartic and the Satake compactification of the moduli space $\sA_{2,3}$. We also thank Noam Elkies for various remarks and suggestions that have made it into the paper, in particular the cases $q=16,19$ in Remark~\ref{R:zeta_small_q}.
Finally, we are grateful to an anonymous referee for pointing out various historical and very readable references to Baker and Coble.

\appendix
\section{Formulae}\label{Appendix}

In this appendix we provide the most important polynomial expressions in computer-readable form. They are expressed in a form that is readily readable by Magma, but with a slight amount of editing it should be possible to make it readable to any other computer algebra system. See
\cite{BruinNasserden17e} for the same data as a plain text file.
\begin{tiny}
%\verbatiminput{burkhardtquartic.magma}
\begin{verbatim}
k:=Integers();
Py<y0,y1,y2,y3,y4>:=PolynomialRing(k,5);
Pt<t1,t2,t3>:=PolynomialRing(k,3);
Ka<a1,a2,a3,a4>:=FunctionField(k,4);
KaX<X>:=PolynomialRing(Ka);
//the defining homogeneous equation of the model of the Burkhardt Quartic we use:
B:=y0^4+y0*y1^3+y0*y2^3+y0*y3^3+3*y1*y2*y3*y4+y0*y4^3;

//A description of a birational map A^3->B given by (y0:y1:y2:y3:y4)
//as polynomials in (t1,t2,t3)
phi:=[t1^3-3*t1^2*t3-3*t1*t2^2-3*t1*t2*t3-t2^3-1,
-t1^3+3*t1^2*t3-3*t1*t3^2+t2^3+1,
-t1^4+t1^3*t2+3*t1^3*t3-3*t1^2*t2*t3-3*t1^2*t3^2-2*t1*t2^3-3*t1*t2^2*t3+t1-t2^4-t2,
-t1^4+4*t1^3*t3+3*t1^2*t2^2+3*t1^2*t2*t3-3*t1^2*t3^2+t1*t2^3-3*t1*t2^2*t3-3*t1*t2*t3^2+t1-t2^3*t3-t3,
-t1^4-t1^3*t2+2*t1^3*t3+3*t1^2*t2*t3+t1*t2^3+3*t1*t2^2*t3+t1+t2^4+t2^3*t3+t2+t3];

//A list of descriptions of a birational map B->P^3
//given as a list of lists (t0:t1:t2:t3) as homogeneous polynomials in
//(y0:y1:y2:y3:y4)
psis:=[[y0^3-y0^2*y1+y0*y1^2,
-y0^2*y3-y0^2*y4+y0*y1*y2,
y0^2*y2-y0*y1*y2+y0*y1*y3+y0*y1*y4,
-y0^2*y4+y0*y1*y2-y0*y1*y3+y1^2*y3
],[
3*y0*y2^2*y4-3*y0*y3^2*y4+3*y0*y3*y4^2-3*y0*y4^3-3*y1*y2^2*y4-3*y1*y2*y3*y4-3*y1*y2*y4^2,
-3*y0^3*y4-3*y0^2*y1*y4-3*y2^3*y4-3*y2^2*y3*y4-3*y2^2*y4^2,
3*y0^2*y1*y4+3*y0*y1^2*y4+3*y2^3*y4-3*y2*y3^2*y4+3*y2*y3*y4^2-3*y2*y4^3,
y0^3*y2+y0^3*y3-2*y0^3*y4-3*y0^2*y1*y4+y1^3*y2+y1^3*y3+y1^3*y4+y2^4+y2^3*y3-2*y2^3*y4-
3*y2^2*y4^2+y2*y3^3+y2*y4^3+y3^4-2*y3^3*y4+3*y3^2*y4^2-2*y3*y4^3+y4^4
],[
3*y0^2*y2*y4-3*y0*y1*y2*y4+3*y1^2*y2*y4,
-3*y0*y2*y3*y4-3*y0*y2*y4^2+3*y1*y2^2*y4,
3*y0*y2^2*y4-3*y1*y2^2*y4+3*y1*y2*y3*y4+3*y1*y2*y4^2,
-y0^3*y1-3*y0*y2*y4^2-y1^4-y1*y2^3+3*y1*y2^2*y4-3*y1*y2*y3*y4-y1*y3^3-y1*y4^3
],[
-y0^2*y2*y3-y0^2*y2*y4-y0^2*y3^2+y0^2*y3*y4-y0^2*y4^2-y0*y1*y2^2+y0*y1*y3^2-y0*y1*y3*y4+
y0*y1*y4^2,
y0^2*y1^2+y0*y1^3+y0*y2*y3^2+2*y0*y2*y3*y4+y0*y2*y4^2+y0*y3^3+y0*y4^3+3*y1*y2*y3*y4,
y0^3*y1-y0*y1^3-y0*y2^2*y3-y0*y2^2*y4-y0*y2*y3^2+y0*y2*y3*y4-y0*y2*y4^2-3*y1*y2*y3*y4,
y0^2*y1^2+y0*y1^3+y0*y2*y3*y4+y0*y2*y4^2+y0*y3^2*y4-y0*y3*y4^2+y0*y4^3-y1*y2^2*y3+
3*y1*y2*y3*y4+y1*y3^3-y1*y3^2*y4+y1*y3*y4^2]];

//Below we give 4 lists [H,lambda,G], where lambda is a rational function in (a1,...,a4)
//and H,G are polynomials in X with coefficients that are rational functions in (a1,...,a4).
//The expression G^2+4*lambda*H^3 yields the same sextic in X for each triple [H,lambda,G].
//When (1:a1:a2:a3:a4) is a point on B that does not lie in a j-plane and has a4 != 0
//then this corresponds to the 3-torsion point corresponding to the j-plane y[0]=y[i]=0
//(for i=1,...,4, in the order given)
//(The article describes H,G as homogeneous forms in (x,z), of degrees 2 and 3 respectively.
//Here we set (x,z)=(X,1) to get a more compact representation)
HLGs:=[[(a1^3*a3*a4^3+a1^2*a2^2*a4^5+a1^2*a2^2*a4^2+2*a1*a2*a3^2*a4^4-a1*a2*a3^2*a4-a2^3*a3+
a3^4*a4^3)*X^2+(a1^4*a4^4+2*a1^2*a2*a3*a4^5-a1^2*a2*a3*a4^2+2*a1*a2^3*a4^4+a1*a2^3*a4+
2*a1*a3^3*a4^4+a1*a3^3*a4+2*a2^2*a3^2*a4^3+2*a2^2*a3^2)*X+a1^3*a2*a4^3+a1^2*a3^2*a4^5+
a1^2*a3^2*a4^2+2*a1*a2^2*a3*a4^4-a1*a2^2*a3*a4+a2^4*a4^3-a2*a3^3,
(-a4^3-1)/(a1^6*a4^6-6*a1^4*a2*a3*a4^4-2*a1^3*a2^3*a4^3-2*a1^3*a3^3*a4^3+9*a1^2*a2^2*a3^2*a4^2+
6*a1*a2^4*a3*a4+6*a1*a2*a3^4*a4+a2^6+2*a2^3*a3^3+a3^6),
(a1^6*a4^6+3*a1^4*a2*a3*a4^7-3*a1^4*a2*a3*a4^4+2*a1^3*a2^3*a4^9+4*a1^3*a2^3*a4^6+3*a1^3*a3^3*a4^6
+a1^3*a3^3*a4^3+6*a1^2*a2^2*a3^2*a4^8+3*a1^2*a2^2*a3^2*a4^5+6*a1^2*a2^2*a3^2*a4^2-
3*a1*a2^4*a3*a4^4+3*a1*a2^4*a3*a4+6*a1*a2*a3^4*a4^7+a2^6-3*a2^3*a3^3*a4^3-a2^3*a3^3+
2*a3^6*a4^6+a3^6*a4^3)/(a1^3*a4^3-3*a1*a2*a3*a4-a2^3-a3^3)*X^3+(3*a1^5*a2*a4^8+
3*a1^5*a2*a4^5+6*a1^4*a3^2*a4^7+6*a1^4*a3^2*a4^4+6*a1^3*a2^2*a3*a4^9+6*a1^3*a2^2*a3*a4^6+
6*a1^2*a2^4*a4^8+9*a1^2*a2^4*a4^5+3*a1^2*a2^4*a4^2+12*a1^2*a2*a3^3*a4^8+3*a1^2*a2*a3^3*a4^5-
9*a1^2*a2*a3^3*a4^2+12*a1*a2^3*a3^2*a4^7+9*a1*a2^3*a3^2*a4^4-3*a1*a2^3*a3^2*a4+6*a1*a3^5*a4^7+
6*a1*a3^5*a4^4-3*a2^5*a3*a4^3-3*a2^5*a3+6*a2^2*a3^4*a4^6+9*a2^2*a3^4*a4^3+
3*a2^2*a3^4)/(a1^3*a4^3-3*a1*a2*a3*a4-a2^3-a3^3)*X^2+(3*a1^5*a3*a4^8+3*a1^5*a3*a4^5+
6*a1^4*a2^2*a4^7+6*a1^4*a2^2*a4^4+6*a1^3*a2*a3^2*a4^9+6*a1^3*a2*a3^2*a4^6+12*a1^2*a2^3*a3*a4^8
+3*a1^2*a2^3*a3*a4^5-9*a1^2*a2^3*a3*a4^2+6*a1^2*a3^4*a4^8+9*a1^2*a3^4*a4^5+3*a1^2*a3^4*a4^2+
6*a1*a2^5*a4^7+6*a1*a2^5*a4^4+12*a1*a2^2*a3^3*a4^7+9*a1*a2^2*a3^3*a4^4-3*a1*a2^2*a3^3*a4+
6*a2^4*a3^2*a4^6+9*a2^4*a3^2*a4^3+3*a2^4*a3^2-3*a2*a3^5*a4^3-3*a2*a3^5)/(a1^3*a4^3-
3*a1*a2*a3*a4-a2^3-a3^3)*X+(a1^6*a4^6+3*a1^4*a2*a3*a4^7-3*a1^4*a2*a3*a4^4+3*a1^3*a2^3*a4^6
+a1^3*a2^3*a4^3+2*a1^3*a3^3*a4^9+4*a1^3*a3^3*a4^6+6*a1^2*a2^2*a3^2*a4^8+3*a1^2*a2^2*a3^2*a4^5
+6*a1^2*a2^2*a3^2*a4^2+6*a1*a2^4*a3*a4^7-3*a1*a2*a3^4*a4^4+3*a1*a2*a3^4*a4+2*a2^6*a4^6+
a2^6*a4^3-3*a2^3*a3^3*a4^3-a2^3*a3^3+a3^6)/(a1^3*a4^3-3*a1*a2*a3*a4-a2^3-a3^3)
],[
a1*a4*X^2+a2*X-a3,
-a1^3*a4^9-a1^3*a4^6+3*a1*a2*a3*a4^7+3*a1*a2*a3*a4^4+a2^3*a4^6+a2^3*a4^3+a3^3*a4^6+a3^3*a4^3,
(2*a1^3*a4^6+a1^3*a4^3-3*a1*a2*a3*a4^4+a2^3-a3^3*a4^3)*X^3+(3*a1^2*a2*a4^5+3*a1^2*a2*a4^2-
3*a2^2*a3*a4^3-3*a2^2*a3)*X^2+(-3*a1^2*a3*a4^5-3*a1^2*a3*a4^2+3*a2*a3^2*a4^3+3*a2*a3^2)*X-
a1^3*a4^3-3*a1*a2*a3*a4^4-a2^3*a4^3-2*a3^3*a4^3-a3^3
],[
a2*X^2-a3*X-a1*a4,
a1^3*a4^9+a1^3*a4^6-3*a1*a2*a3*a4^7-3*a1*a2*a3*a4^4-a2^3*a4^6-a2^3*a4^3-a3^3*a4^6-a3^3*a4^3,
(a1^3*a4^3+3*a1*a2*a3*a4^4+2*a2^3*a4^3+a2^3+a3^3*a4^3)*X^3+(3*a1^2*a2*a4^5+3*a1^2*a2*a4^2-
3*a2^2*a3*a4^3-3*a2^2*a3)*X^2+(-3*a1^2*a3*a4^5-3*a1^2*a3*a4^2+3*a2*a3^2*a4^3+3*a2*a3^2)*X-
2*a1^3*a4^6-a1^3*a4^3+3*a1*a2*a3*a4^4+a2^3*a4^3-a3^3
],[
(a1*a2^2*a4^3+a1*a2^2)*X^2+(a1^3*a4^2+a1*a2*a3*a4^3-2*a1*a2*a3+a2^3*a4^2+a3^3*a4^2)*X+
a1*a3^2*a4^3+a1*a3^2,
(-a1^3*a4^3+3*a1*a2*a3*a4+a2^3+a3^3)/(a1^6+6*a1^4*a2*a3*a4+2*a1^3*a2^3+2*a1^3*a3^3+
9*a1^2*a2^2*a3^2*a4^2+6*a1*a2^4*a3*a4+6*a1*a2*a3^4*a4+a2^6+2*a2^3*a3^3+a3^6),
(a1^6*a4^3+6*a1^4*a2*a3*a4^4+2*a1^3*a2^3*a4^6+5*a1^3*a2^3*a4^3+a1^3*a2^3+2*a1^3*a3^3*a4^3+
9*a1^2*a2^2*a3^2*a4^5+3*a1*a2^4*a3*a4^4-3*a1*a2^4*a3*a4+6*a1*a2*a3^4*a4^4-a2^6+a2^3*a3^3*a4^3
-a2^3*a3^3+a3^6*a4^3)/(a1^3+3*a1*a2*a3*a4+a2^3+a3^3)*X^3+(3*a1^5*a2*a4^5+3*a1^5*a2*a4^2+
3*a1^3*a2^2*a3*a4^6-3*a1^3*a2^2*a3+3*a1^2*a2^4*a4^5+3*a1^2*a2^4*a4^2+3*a1^2*a2*a3^3*a4^5+
3*a1^2*a2*a3^3*a4^2+9*a1*a2^3*a3^2*a4^4+9*a1*a2^3*a3^2*a4+3*a2^5*a3*a4^3+3*a2^5*a3+
3*a2^2*a3^4*a4^3+3*a2^2*a3^4)/(a1^3+3*a1*a2*a3*a4+a2^3+a3^3)*X^2+(-3*a1^5*a3*a4^5-
3*a1^5*a3*a4^2-3*a1^3*a2*a3^2*a4^6+3*a1^3*a2*a3^2-3*a1^2*a2^3*a3*a4^5-3*a1^2*a2^3*a3*a4^2-
3*a1^2*a3^4*a4^5-3*a1^2*a3^4*a4^2-9*a1*a2^2*a3^3*a4^4-9*a1*a2^2*a3^3*a4-3*a2^4*a3^2*a4^3-
3*a2^4*a3^2-3*a2*a3^5*a4^3-3*a2*a3^5)/(a1^3+3*a1*a2*a3*a4+a2^3+a3^3)*X+(-a1^6*a4^3-
6*a1^4*a2*a3*a4^4-2*a1^3*a2^3*a4^3-2*a1^3*a3^3*a4^6-5*a1^3*a3^3*a4^3-a1^3*a3^3-
9*a1^2*a2^2*a3^2*a4^5-6*a1*a2^4*a3*a4^4-3*a1*a2*a3^4*a4^4+3*a1*a2*a3^4*a4-a2^6*a4^3-
a2^3*a3^3*a4^3+a2^3*a3^3+a3^6)/(a1^3+3*a1*a2*a3*a4+a2^3+a3^3)]];

//Below is a triple (H,lambda,G) such that G^2+4*lambda*H^3 is -3*F, where F is the sextic
//defined by HLGs above. This triple corresponds to the 3-torsion points that the j-plane
//y0+...+y4=y0+y4=0 marks if (1:a1:a2:a3:a4) is a point on the Burkhardt quartic. Note that
//the identity of the sextics holds regardless of whether (1:a1:a2:a3:a4) satisfy the Burkhardt
//relation. We do need the Burkhardt relation to get the other cyclic order 3 subgroups defined
//over the base field.
HLGdual:=[(a1^2*a4^2+a1*a2*a4^3-a1*a2*a4^2-a1*a2*a4-a1*a3*a4^2+a2^2+a2*a3*a4+a3^2*a4^2)*X^2+(-a1^2*a4^3+
a1^2*a4^2+a1*a2*a4^3+a1*a2*a4+a1*a3*a4^3+a1*a3*a4+a2^2*a4^2-a2^2*a4-2*a2*a3+a3^2*a4^2-
a3^2*a4)*X+a1^2*a4^2-a1*a2*a4^2+a1*a3*a4^3-a1*a3*a4^2-a1*a3*a4+a2^2*a4^2+a2*a3*a4+a3^2,
(-3*a1^2*a4^4+3*a1^2*a4^3-3*a1^2*a4^2-3*a1*a2*a4^3+3*a1*a2*a4^2-3*a1*a2*a4-3*a1*a3*a4^3+
3*a1*a3*a4^2-3*a1*a3*a4-3*a2^2*a4^2+3*a2^2*a4-3*a2^2+3*a2*a3*a4^2-3*a2*a3*a4+3*a2*a3-
3*a3^2*a4^2+3*a3^2*a4-3*a3^2)/(a1^2*a4^4+2*a1^2*a4^3+a1^2*a4^2-2*a1*a2*a4^3-4*a1*a2*a4^2-
2*a1*a2*a4-2*a1*a3*a4^3-4*a1*a3*a4^2-2*a1*a3*a4+a2^2*a4^2+2*a2^2*a4+a2^2+2*a2*a3*a4^2+
4*a2*a3*a4+2*a2*a3+a3^2*a4^2+2*a3^2*a4+a3^2),
(-3*a1^4*a4^5+3*a1^4*a4^4-6*a1^3*a2*a4^6+6*a1^3*a2*a4^5-3*a1^3*a2*a4^4-3*a1^3*a2*a4^3+
6*a1^3*a3*a4^5-3*a1^3*a3*a4^4+3*a1^3*a3*a4^3+6*a1^2*a2^2*a4^6+6*a1^2*a2^2*a4^4+6*a1^2*a2^2*a4^2+
3*a1^2*a2*a3*a4^6-3*a1^2*a2*a3*a4^5+6*a1^2*a2*a3*a4^4+6*a1^2*a2*a3*a4^3-6*a1^2*a2*a3*a4^2-
6*a1^2*a3^2*a4^5+6*a1^2*a3^2*a4^4-6*a1^2*a3^2*a4^3-6*a1*a2^3*a4^4+6*a1*a2^3*a4^3-3*a1*a2^3*a4^2-
3*a1*a2^3*a4-3*a1*a2^2*a3*a4^5+3*a1*a2^2*a3*a4^4-6*a1*a2^2*a3*a4^3-6*a1*a2^2*a3*a4^2+
6*a1*a2^2*a3*a4+9*a1*a2*a3^2*a4^5-3*a1*a2*a3^2*a4^4-6*a1*a2*a3^2*a4^3+6*a1*a2*a3^2*a4^2+
3*a1*a3^3*a4^5-3*a1*a3^3*a4^4+6*a1*a3^3*a4^3-3*a2^4*a4+3*a2^4-6*a2^3*a3*a4^2+3*a2^3*a3*a4-
3*a2^3*a3-6*a2^2*a3^2*a4^3+6*a2^2*a3^2*a4^2-6*a2^2*a3^2*a4-3*a2*a3^3*a4^4+3*a2*a3^3*a4^3-
6*a2*a3^3*a4^2+3*a3^4*a4^4-3*a3^4*a4^3)/(a1*a4^2+a1*a4-a2*a4-a2-a3*a4-a3)*X^3+(6*a1^4*a4^6
-6*a1^4*a4^5+6*a1^4*a4^4+3*a1^3*a2*a4^7-9*a1^3*a2*a4^6+12*a1^3*a2*a4^5-9*a1^3*a2*a4^4+
3*a1^3*a2*a4^3-6*a1^3*a3*a4^6+12*a1^3*a3*a4^5-12*a1^3*a3*a4^4+6*a1^3*a3*a4^3+3*a1^2*a2^2*a4^6-
15*a1^2*a2^2*a4^5+18*a1^2*a2^2*a4^4-15*a1^2*a2^2*a4^3+3*a1^2*a2^2*a4^2+9*a1^2*a2*a3*a4^6-
9*a1^2*a2*a3*a4^5+9*a1^2*a2*a3*a4^3-9*a1^2*a2*a3*a4^2+6*a1^2*a3^2*a4^6-12*a1^2*a3^2*a4^5+
18*a1^2*a3^2*a4^4-12*a1^2*a3^2*a4^3+6*a1^2*a3^2*a4^2+12*a1*a2^3*a4^5-6*a1*a2^3*a4^4+
12*a1*a2^3*a4^3+6*a1*a2^3*a4+3*a1*a2^2*a3*a4^5+3*a1*a2^2*a3*a4^4+3*a1*a2^2*a3*a4^2+
3*a1*a2^2*a3*a4+6*a1*a2*a3^2*a4^5-12*a1*a2*a3^2*a4^4+6*a1*a2*a3^2*a4^2-12*a1*a2*a3^2*a4+
6*a1*a3^3*a4^5-12*a1*a3^3*a4^4+12*a1*a3^3*a4^3-6*a1*a3^3*a4^2-6*a2^4*a4^3+6*a2^4*a4^2-6*a2^4*a4
-3*a2^3*a3*a4^4+3*a2^3*a3*a4^3-12*a2^3*a3*a4^2+9*a2^3*a3*a4-9*a2^3*a3+9*a2^2*a3^2*a4^4-
9*a2^2*a3^2*a4^3+18*a2^2*a3^2*a4^2-9*a2^2*a3^2*a4+9*a2^2*a3^2+12*a2*a3^3*a4^3-12*a2*a3^3*a4^2+
12*a2*a3^3*a4+6*a3^4*a4^4-6*a3^4*a4^3+6*a3^4*a4^2)/(a1*a4^2+a1*a4-a2*a4-a2-a3*a4-a3)*X^2+
(6*a1^4*a4^6-6*a1^4*a4^5+6*a1^4*a4^4-6*a1^3*a2*a4^6+12*a1^3*a2*a4^5-12*a1^3*a2*a4^4+
6*a1^3*a2*a4^3+3*a1^3*a3*a4^7-9*a1^3*a3*a4^6+12*a1^3*a3*a4^5-9*a1^3*a3*a4^4+3*a1^3*a3*a4^3+
6*a1^2*a2^2*a4^6-12*a1^2*a2^2*a4^5+18*a1^2*a2^2*a4^4-12*a1^2*a2^2*a4^3+6*a1^2*a2^2*a4^2+
9*a1^2*a2*a3*a4^6-9*a1^2*a2*a3*a4^5+9*a1^2*a2*a3*a4^3-9*a1^2*a2*a3*a4^2+3*a1^2*a3^2*a4^6-
15*a1^2*a3^2*a4^5+18*a1^2*a3^2*a4^4-15*a1^2*a3^2*a4^3+3*a1^2*a3^2*a4^2+6*a1*a2^3*a4^5-
12*a1*a2^3*a4^4+12*a1*a2^3*a4^3-6*a1*a2^3*a4^2+6*a1*a2^2*a3*a4^5-12*a1*a2^2*a3*a4^4+
6*a1*a2^2*a3*a4^2-12*a1*a2^2*a3*a4+3*a1*a2*a3^2*a4^5+3*a1*a2*a3^2*a4^4+3*a1*a2*a3^2*a4^2+
3*a1*a2*a3^2*a4+12*a1*a3^3*a4^5-6*a1*a3^3*a4^4+12*a1*a3^3*a4^3+6*a1*a3^3*a4+6*a2^4*a4^4-
6*a2^4*a4^3+6*a2^4*a4^2+12*a2^3*a3*a4^3-12*a2^3*a3*a4^2+12*a2^3*a3*a4+9*a2^2*a3^2*a4^4-
9*a2^2*a3^2*a4^3+18*a2^2*a3^2*a4^2-9*a2^2*a3^2*a4+9*a2^2*a3^2-3*a2*a3^3*a4^4+3*a2*a3^3*a4^3-
12*a2*a3^3*a4^2+9*a2*a3^3*a4-9*a2*a3^3-6*a3^4*a4^3+6*a3^4*a4^2-6*a3^4*a4)/(a1*a4^2+a1*a4-
a2*a4-a2-a3*a4-a3)*X+(-3*a1^4*a4^5+3*a1^4*a4^4+6*a1^3*a2*a4^5-3*a1^3*a2*a4^4+3*a1^3*a2*a4^3
-6*a1^3*a3*a4^6+6*a1^3*a3*a4^5-3*a1^3*a3*a4^4-3*a1^3*a3*a4^3-6*a1^2*a2^2*a4^5+6*a1^2*a2^2*a4^4-
6*a1^2*a2^2*a4^3+3*a1^2*a2*a3*a4^6-3*a1^2*a2*a3*a4^5+6*a1^2*a2*a3*a4^4+6*a1^2*a2*a3*a4^3-
6*a1^2*a2*a3*a4^2+6*a1^2*a3^2*a4^6+6*a1^2*a3^2*a4^4+6*a1^2*a3^2*a4^2+3*a1*a2^3*a4^5-
3*a1*a2^3*a4^4+6*a1*a2^3*a4^3+9*a1*a2^2*a3*a4^5-3*a1*a2^2*a3*a4^4-6*a1*a2^2*a3*a4^3+
6*a1*a2^2*a3*a4^2-3*a1*a2*a3^2*a4^5+3*a1*a2*a3^2*a4^4-6*a1*a2*a3^2*a4^3-6*a1*a2*a3^2*a4^2+
6*a1*a2*a3^2*a4-6*a1*a3^3*a4^4+6*a1*a3^3*a4^3-3*a1*a3^3*a4^2-3*a1*a3^3*a4+3*a2^4*a4^4-
3*a2^4*a4^3-3*a2^3*a3*a4^4+3*a2^3*a3*a4^3-6*a2^3*a3*a4^2-6*a2^2*a3^2*a4^3+6*a2^2*a3^2*a4^2-
6*a2^2*a3^2*a4-6*a2*a3^3*a4^2+3*a2*a3^3*a4-3*a2*a3^3-3*a3^4*a4+3*a3^4)/(a1*a4^2+a1*a4-a2*a4-
a2-a3*a4-a3)];

//magma code to verify that the expressions given indeed satisfy the relations claimed,
L:=[[Evaluate(p,phi):p in q]:q in psis];
assert {[Pt|c/l[1]: c in l]:l in L} eq{[1,t1,t2,t3]};
assert Evaluate(B,phi) eq 0;
V:={c[3]^2+4*c[2]*c[1]^3:c in HLGs};
assert #V eq 1;
F:=Rep(V);
G:=c[3]^2+4*c[2]*c[1]^3 where c:=HLGdual;
assert G/F eq -3;
\end{verbatim}
\end{tiny}
\end{document}